\documentclass[12pt]{article}

\usepackage{vmargin}
\usepackage{fancyhdr}
\usepackage{ifthen}
\usepackage{graphicx} 
\usepackage{indentfirst}
\usepackage[english]{babel}

\usepackage{fouriernc}
\DeclareMathAlphabet{\mathcal}{OMS}{cmsy}{m}{n} 

\usepackage{amsmath}   
\usepackage{amssymb}   
\usepackage{amsfonts}
\usepackage{amsthm}
\usepackage{mathrsfs}

\usepackage{pst-node} 
\usepackage{pstricks}
\usepackage{pstricks-add}
\usepackage{lscape}
\usepackage{rotating}
\usepackage{xypic}

\makeatletter
\def\enumhook{}

\def\itemhook{}

\def\descripthook{}
\def\enumerate{%
  \ifnum \@enumdepth >\thr@@\@toodeep\else
    \advance\@enumdepth\@ne
    \edef\@enumctr{enum\romannumeral\the\@enumdepth}%
      \expandafter
      \list
        \csname label\@enumctr\endcsname
        {\usecounter\@enumctr\def\makelabel##1{\hss\llap{##1}}%
          \enumhook \csname enumhook\romannumeral\the\@enumdepth\endcsname}%
  \fi}
\def\itemize{%
  \ifnum \@itemdepth >\thr@@\@toodeep\else
    \advance\@itemdepth\@ne
    \edef\@itemitem{labelitem\romannumeral\the\@itemdepth}%
    \expandafter
    \list
      \csname\@itemitem\endcsname
      {\def\makelabel##1{\hss\llap{##1}}%
        \itemhook \csname itemhook\romannumeral\the\@itemdepth\endcsname}%
  \fi}

\renewcommand{\itemhook}{\setlength{\topsep}{1pt}%
\setlength{\itemsep}{-1pt}}

\renewcommand{\enumhook}{\setlength{\topsep}{1pt}%
\setlength{\itemsep}{-1pt}}
\makeatother

\usepackage[colorlinks=true,linkcolor=magenta,citecolor=blue]{hyperref}

\setpapersize{A4}
\setmarginsrb{33mm}{20mm}{33mm}{25mm}{10mm}{6mm}{0mm}{10mm}
\fancypagestyle{plain}{%
\fancyhf{} 
\lfoot[\thepage]{}
\cfoot[]{}
\rfoot[]{\thepage}
}

\newcommand\ie{\emph{i.e.}}

\newcommand{\sk}{\smallskip}
\newcommand{\mk}{\medskip}
\newcommand{\bk}{\bigskip}
\renewcommand{\emptyset}{\ensuremath{\varnothing}}     

\newcommand{\G}{\mathbf{G}}
\newcommand{\B}{\mathbf{B}}
\newcommand{\U}{\mathbf{U}}

\newcommand{\T}{\mathbf{T}}
\renewcommand{\P}{\mathbf{P}}
\renewcommand{\L}{\mathbf{L}}

\newcommand{\X}{\mathrm{X}}
\newcommand{\Y}{\mathrm{Y}}
\newcommand{\Hc}{\mathrm{H}_c}
\newcommand{\Rgc}{\mathrm{R}\Gamma_c}
\newcommand{\tRgc}{\widetilde{\mathrm{R}}\Gamma_c}
\newcommand{\gRgc}{\mathrm{G}\Gamma_c}

\newcommand\Ext{\mathrm{Ext}}
\newcommand\Ho{\mathrm{Ho}}
\newcommand\Res{\mathrm{Res}}
\newcommand\Ind{\mathrm{Ind}}
\newcommand\Aut{\mathrm{Aut}}
\newcommand\Stab{\mathrm{Stab}}
\newcommand\Br{\mathrm{Br}}
\newcommand\iso{\ \widetilde \to\ }

\newcommand\mMod{\text{-}\mathrm{mod}}
\newcommand\mstab{\text{-}\mathrm{stab}}
\newcommand\mperf{\text{-}\mathrm{perf}}
\newcommand\mproj{\text{-}\mathrm{proj}}
\newcommand\mPROJ{\text{-}\mathrm{Proj}}
\newcommand\mMOD{\text{-}\mathrm{Mod}}
\newcommand\mlocfin{\text{-}\mathrm{locfin}}
\newcommand\Comp{\mathrm{Comp}}
\newcommand\Hom{\mathrm{Hom}}

\newcommand\opp{\mathrm{opp}}
\newcommand\cone{\mathrm{cone}}
\newcommand\hocolim{\mathrm{hocolim}}

\newlength{\leftlength}
\newlength{\rightlength}
\newlength{\calculskip}
\newcommand{\calculvskip}[1]{%
  \ifthenelse{#1 = 0}{\setlength{\calculskip}{0pt}}{}%
  \ifthenelse{#1 = 1}{\setlength{\calculskip}{\smallskipamount}}{}%
  \ifthenelse{#1 = 2}{\setlength{\calculskip}{\medskipamount}}{}%
  \ifthenelse{#1 = 3}{\setlength{\calculskip}{\bigskipamount}}{}%
  \ifthenelse{#1 = 4}{\setlength{\calculskip}{1cm}}{}%
  \vskip\calculskip
}

\newcommand{\leftcentersright}[4][2]{%
        \settowidth{\leftlength}{#2}%
        \settowidth{\rightlength}{#4}%
        \calculvskip{#1}
        \noindent#2\hskip-\leftlength%
        \hfill#3\hfill
        \mbox{}\hskip-\rightlength#4%
        \vskip\calculskip%
        }

\newcommand{\centers}[2][2]{\leftcentersright[#1]{}{#2}{}}
\newcommand{\leftcenters}[3][2]{\leftcentersright[#1]{#2}{#3}{}}

\makeatletter
\def\svhline{%
  \noalign{\ifnum0=`}\fi\hrule \@height2\arrayrulewidth \futurelet
   \reserved@a\@xhline}
   
\def\hlinewd#1{%
\noalign{\ifnum0=`}\fi\hrule \@height #1 %
\futurelet\reserved@a\@xhline}
 \makeatother

\numberwithin{equation}{section}

\newtheorem{prop}[equation]{Proposition}  
\newtheorem{thm}[equation]{Theorem}
\newtheorem{lem}[equation]{Lemma}  
\newtheorem{cor}[equation]{Corollary}
\newtheorem{conj}[equation]{Conjecture}

\theoremstyle{definition}
\newtheorem{exemple}[equation]{Example}  
\newtheorem{rmk}[equation]{Remark}

\begin{document}

\title{Coxeter orbits and Brauer trees III}

\author{Olivier Dudas\footnote{The first author is supported by the EPSRC, Project No EP/H026568/1 and by Magdalen College, Oxford} \and Rapha\"el Rouquier}

\maketitle

\begin{abstract} 
This article is the final one of a series of articles (cf. \cite{Du3,Du4})
on certain blocks of
modular representations of finite groups of Lie type and the associated geometry.
We prove the conjecture of Brou\'e on derived equivalences induced by the complex
of cohomology of Deligne-Lusztig varieties in the case of Coxeter elements whenever the defining characteristic is good. We also
prove a conjecture of Hi\ss, L\"ubeck and Malle on the Brauer trees of the corresponding
blocks. As a consequence, we determine the Brauer trees (in particular, the decomposition matrix) of the principal $\ell$-block of $E_7(q)$ when $\ell \mid \Phi_{18}(q)$ and $E_8(q)$ when $\ell \mid \Phi_{18}(q)$ or $\ell \mid \Phi_{30}(q)$.
\end{abstract}

\section*{Introduction}

This article is the final one of a series of articles on certain blocks of
modular representations of finite groups of Lie type and the associated geometry.
This series gives the first instance of use of the mod-$\ell$ cohomology
of Deligne-Lusztig varieties to determine new decomposition matrices of principal blocks for finite groups
of Lie type. 

\sk

In the first two articles \cite{Du3,Du4}, the first author studied the integral $\ell$-adic
cohomology of Deligne-Lusztig varieties associated with Coxeter elements. For
suitable primes $\ell$, Brou\'e \cite{BMa2} has conjectured that the complex of cohomology provides a solution to his abelian defect group conjecture for the principal block.
On the other hand, Hi\ss, L\"ubeck and Malle have conjectured that the Brauer tree
of the block can be recovered from the rational $\ell$-adic cohomology, endowed with
the action of the Frobenius \cite{HLM}. In \cite{Du3,Du4}, the relation between these conjectures was
studied, and Brou\'e's conjecture was shown to hold for Coxeter elements,
with the some possible exceptions of types $E_7$ and $E_8$. These are also the cases for
which the conjecture of Hi\ss, L\"ubeck and Malle was still open. We give here a
general proof of that conjecture and, as a consequence of \cite{Du3,Du4}, of Brou\'e's conjecture. The new ingredient is the study of the
complex of cohomology, and the corresponding functor, in suitable stable categories. This
requires proving first some finiteness properties for the complex, when viewed as
a complex of $\ell$-permutation modules with an action of the Frobenius endomorphism.
A key input from \cite{Du3,Du4} is the property of the mod-$\ell$ cohomology associated
with certain "minimal"
eigenvalues of Frobenius to be concentrated in one degree. As a consequence,
we determine the Brauer trees for the finite reductive groups of type $E_7$ and $E_8$, for
primes dividing the cyclotomic polynomial associated with the Coxeter number. We also
obtain the previously unknown planar embeddings for the trees associated with the groups
of type ${}^2 F_4$ and $F_4$. From \cite{Du3,Du4}, we deduce that Brou\'e's conjecture holds
in the case of Coxeter elements whenever $p$ is a good prime. Note that David Craven has recently proposed a conjecture for the Brauer trees
of all unipotent blocks of finite groups of Lie type, together with
a conjecture for the perversity function associated with the
equivalences predicted by Brou\'e \cite{Cra1,Cra2}.

\bk

Let us describe in detail the structure of this article. 
In the first section, we start with an analysis of \emph{good} algebras, \ie,
algebras for which
every bounded complex with finite-dimensional cohomology is quasi-isomorphic
to a  complex with finite-dimensional components. Given a group, its group algebra is good over arbitrary finite fields if
and only if the group is good, \ie, the cohomology of the group and
its profinite completion agree for any finite module. Consider a group $\Upsilon$ with a finite normal subgroup $H$ such that $\Upsilon/H$ is good and let $k$ be a finite field. We show that a complex of $k\Upsilon$-modules whose restrictions to $kH$ is perfect is quasi-isomorphic to a bounded complex of $k\Upsilon$-modules whose restrictions to $kH$ are finitely generated and projective.
We apply this to the complex of  cohomology of an algebraic variety $X$
acted on by a monoid $\Upsilon^+$ 
acting by invertible transformations of the \'etale site, where $\Upsilon$ is the group associated with $\Upsilon^+$. Let $\ell$ be a prime invertible on $X$. We show that the complex of mod-$\ell$
cohomology of $X$ can be represented by a bounded complex of
finite $\Upsilon$-modules which are direct summands of permutation
modules for $H$ (building on \cite{Ri5,Rou}). This is motivated by the geometric form of Brou\'e's abelian
defect conjecture: given $G$ a finite group of Lie type and
$B$ a block with abelian defect group $D$, there should exist a Deligne-Lusztig
variety $X$ acted on by $H=G\times C_G(D)^\opp\lhd \Upsilon^+$, where
$\Upsilon / H$ is a  braid group associated with a complex reflection group, and such that the complex of mod-$\ell$ cohomology of $X$ provides a derived equivalence between the principal blocks of $G$ and $N_G(D)$, the action of which arises from $\Upsilon/G$. The results of the first section apply, at least when
the complex reflection group does not have exceptional irreducible
components of dimension $\ge 3$, as the corresponding braid group is good.

\sk

In the second section, we consider a reductive connected algebraic group $\G$
endowed with an endomorphism $F$ a power of which is a Frobenius endomorphism. We
study the complex of cohomology of the Deligne-Lusztig
variety associated with a parabolic subgroup with an $F$-stable Levi
complement $\mathbf{L}$. Under the assumption that the Sylow $\ell$-subgroups of $\G^F$ are cyclic, and that $\mathbf{L}$ is the centraliser of one of them, we study the generalised eigenspaces of the Frobenius on the complex of cohomology (in the derived equivalence situation, they correspond to the images of the simple modules for $N_G(D)$).  We determine their class in the stable category of $\G^F$.

\sk

The third section is devoted to mod-$\ell$ representations of $\mathbf{G}^F = \G(\mathbb{F}_q)$,
where $\G$ is simple and the multiplicative
order of $q$ modulo $\ell$ is the Coxeter number of $(\mathbf{G},F)$ (with a suitable modification for Ree and Suzuki groups). We show that the knowledge of the stable
equivalence
induced by the Coxeter Deligne-Lusztig variety, together with the
vanishing results of \cite{Du4}, determine Green's walk around
the Brauer tree of the principal block, as predicted by Hi\ss-L\"ubeck-Malle \cite{HLM}. We also show how to determine the Brauer trees of the non-principal blocks. Finally, we draw the new Brauer
trees for the types ${^2F}_4$, $F_4$, $E_7$ and $E_8$.

\section{Finiteness of complexes of chains}

\subsection{Good algebras}


\noindent \textbf{\thesubsection.1 Locally finite modules.\label{sect:locfinite}} Let $k$ be a field. Given $B$ a $k$-algebra, we denote by $B\mMOD$ the category of left $B$-modules, by $B\mMod$
the category of $B$-modules that are finite-dimensional over $k$ and
by $B\mlocfin$ the category of locally finite $B$-modules, \ie,
$B$-modules which are union of $B$-submodules
in $B\mMod$. These are Serre subcategories of $B\mMOD$. We denote by $B\mPROJ$ (resp. $B\mproj$) the category of projective (resp. finitely generated projective) $B$-modules.

\sk

Given $\mathcal{C}$ an additive category, we denote by
$\mathrm{Comp}(\mathcal{C})$ its category of complexes and by
$\mathrm{Comp}^b(\mathcal{C})$ its subcategory of  bounded complexes. We denote by $\mathrm{Ho}(\mathcal{C})$ the
homotopy category of complexes of $\mathcal{C}$. 

\sk

Assume now $\mathcal{C}$ is an abelian category. Let $C\in\mathrm{Comp}(\mathcal{C})$ and let $n\in\mathbb{Z}$. We put 

\leftcenters{and }{ $\begin{array}[b]{l} \tau_{\le n}C\, =\, \cdots\longrightarrow C^{n-2}\longrightarrow C^{n-1}\longrightarrow \ker d^n\longrightarrow 0 \\[6pt]
\tau_{\ge n}C\, =\, 0\longrightarrow \mathrm{coker} d^{n-1}\longrightarrow C^{n+1}\longrightarrow C^{n+1}\longrightarrow \cdots. \end{array}$}

The derived category of $\mathcal{C}$ will be denoted by $D(\mathcal{C})$. Given $\mathcal{I}$ a subcategory of $\mathcal{C}$,
we denote by $D_{\mathcal{I}}(\mathcal{C})$ the full subcategory of
$D(\mathcal{C})$  of complexes with cohomology  in $\mathcal{I}$. We put $\mathrm{Ho}(B) = \mathrm{Ho}(B\mMOD)$ and $D(B) = D(B\mMOD)$. Recall that an object of $D(B)$ is \emph{perfect} if it is isomorphic to an object of $\mathrm{Comp}^b(B\mproj)$. We refer to \cite[\S 8.1]{Ke} for basic definitions and properties of unbounded derived categories.

\begin{lem}
\label{le:fininlocfin}
The category $D(B\mlocfin)$ is a  triangulated category closed under direct sums and  it is generated by $B\mMod$ as such.
\end{lem}

\begin{proof}
The category $B\mlocfin$ is closed under direct sums, hence $\Comp(B\mlocfin)$
is closed under direct sums. It follows that $D(B\mlocfin)$ is closed under direct
sums and the canonical functor $\Comp(B\mlocfin)\longrightarrow D(B\mlocfin)$ commutes with
direct sums \cite[Lemma 1.5]{BoNee}. Let $C\in\Comp(B\mlocfin)$. We have
$\hocolim_{n\to\infty} \tau_{\le n}(C)\simeq C$, \ie, there is a distinguished triangle

\centers{$\bigoplus \tau_{\le n}(C)\longrightarrow\bigoplus \tau_{\le n}(C)\longrightarrow C\rightsquigarrow.$}

\noindent If $C$ is right bounded, then $\hocolim_{n\to -\infty} \sigma_{\ge n}C\iso C$, where
$\sigma_{\ge n}C=0\longrightarrow C^n\to C^{n+1}\longrightarrow\cdots$ is the subcomplex of $C$ obtained by stupid truncation. It follows that $B\mlocfin$ generates $D(B\mlocfin)$ as
a triangulated category closed under direct sums.

\sk

Since $\hocolim_{n\to-\infty} \sigma_{\ge n}C\simeq M$, we deduce that $M$ is  is in the smallest triangulated subcategory of $D(B\mlocfin)$
closed under direct sums and containing $B\mMod$, and the lemma follows.
\end{proof}

\begin{lem}\label{lem:fgtoloc}
The canonical functor $D^b(B\mMod)\longrightarrow D_{B\mMod}^b(B\mlocfin)$ is an equivalence.
\end{lem}

\begin{proof}
Let $C\in\Comp^b(B\mlocfin)$ with $C^i=0$ for $i>0$, $H^i(C)=0$ for 
$i\not=0$ and $H^0(C)\in B\mMod$. Since $C^0$ is locally finite,
we deduce there is a $B$-submodule $D^0$ of $C^0$ which is finite-dimensional  over $k$ and such that
$D^0+d^{-1}(C^{-1})=C^0$.
We define now $D^{-i}\subset C^{-i}$ by induction on $i$ for
$i\ge 1$. We let $D^{-i}$ be a $B$-submodule of $C^{-i}$ that is 
finite-dimensional over $k$ and such that $d^{-i}(D^{-i})=d^{-i}(C^{-i})
\cap D^{-i+1}$. 
This defines a subcomplex $D$ of $C$ such that $D \hookrightarrow C$ is a quasi-isomorphism.
\sk

We deduce that given $M\in B\mMod$ and $N\in D^b(B\mMod)$, 
the canonical map
$\Hom_{D^b(B\mMod)}(M,N)\longrightarrow \Hom_{D^b(B\mlocfin)}(M,N)$ is an isomorphism.
Since $D^b(B\mMod)$ is generated by $B\mMod$ as a triangulated category, we deduce that the functor of the lemma  is fully faithful. It is then an equivalence, since $B\mMod$ also generates the triangulated category $D_{B\mMod}^b(B\mlocfin)$.\end{proof}


\begin{lem}
\label{le:finiteExt}
The following assertions are equivalent:
\begin{enumerate}
\item[$\mathrm{(1)}$]
the canonical functor $D^b(B\mMod)\longrightarrow D^b_{B\mMod}(B\mMOD)$ is an equivalence
\item[$\mathrm{(2)}$]
the canonical functor $D^b(B\mMod)\longrightarrow D^b_{B\mMod}(B\mMOD)$ is essentially
surjective.
\item[$\mathrm{(3)}$]
given $M,N\in B\mMod$ and given $n\ge 1$,
 the canonical map \linebreak
$\Ext^n_{B\mMod}(M,N)\longrightarrow\Ext^n_{B\mMOD}(M,N)$ is bijective.
\item[$\mathrm{(4)}$]
given $M,N\in B\mMod$ and given $n\ge 2$,
 the canonical map \linebreak
$\Ext^n_{B\mMod}(M,N)\longrightarrow\Ext^n_{B\mMOD}(M,N)$ is surjective.
\end{enumerate}
\end{lem}

\begin{proof} The implication (3)$\Rightarrow$(1) follows from the fact that
$D^b(B\mMod)$ is generated by $B\mMod$ as a triangulated category.
\sk

The implication (4)$\Rightarrow$(3) is proven as in \cite[Exercice 1(a) p.13]{Ser}
by induction on $n$ (the case $n=1$ holds with no assumption). Let $f\in\Ext^{n+1}_{B\mMod}(M,N)$: it is represented
by a long exact sequence

\centers{$0\longrightarrow N\longrightarrow N_1\longrightarrow\cdots\longrightarrow N_{n+1}\longrightarrow M\longrightarrow 0$}

\noindent of objects of $B\mMod$. We have a commutative diagram
with exact rows

\centers{$\xymatrix{\Ext^n_{B\mMod}(M,N_1)\ar[r]\ar[d]_{\rotatebox{90}{$\sim$}} &
\Ext^n_{B\mMod}(M,N_1/N)\ar[r]\ar[d]_{\rotatebox{90}{$\sim$}} & \Ext^{n+1}_{B\mMod}(M,N) \ar[r]
\ar[d]
& \Ext^{n+1}_{B\mMod}(M,N_1)\ar[d] \\
\Ext^n_{B\mMOD}(M,N_1)\ar[r] &
\Ext^n_{B\mMOD}(M,N_1/N)\ar[r] & \Ext^{n+1}_{B\mMOD}(M,N) \ar[r] &
\Ext^{n+1}_{B\mMOD}(M,N_1) \\
}$}

\noindent The image of $f$ in
$\Ext^{n+1}_{B\mMod}(M,N_1)$ vanishes. We deduce that $f$ is the image of a map
$g\in\Ext^n_{B\mMod}(M,N_1/N)$. If $f\not=0$, then $g\not=0$ and,
by induction, the image of $g$ in $\Ext^n_{B\mMOD}(M,N_1/N)$ is not zero. By chasing on the commutative diagram above, 
we deduce that the image of $f$ in $\Ext^{n+1}_{B\mMOD}(M,N)$ is not zero.

\sk

Let us assume now (2). Let $n\ge 2$, let
$f\in\Hom_{D^b(B\mMOD)}(M,N[n])$ and let
$C$ be the cone of $f$. It is the image of an object $D$ of
$D^b(B\mMod)$ and there is a
distinguished triangle $H^{-n}(D)[n]\longrightarrow D\longrightarrow H^{-1}(D)[1]\rightsquigarrow$. This triangle defines a map $M=H^{-1}(D)\longrightarrow N[n]=H^{-n}(D)[n]$ lifting $f$. This
shows (4). Note finally that (1)$\Rightarrow$(2) is trivial.
\end{proof}
We say that $B$ is \emph{good} if it satisfies
any of the equivalent assertions of Lemma \ref{le:finiteExt}.

\begin{lem}
\label{le:finiteindex} Let $A$ be a subalgebra of a $k$-algebra $B$ making $B$ into a
finitely generated
projective $A$-module. Then $A$ is good if and only if $B$ is good.
\end{lem}

\begin{proof} Under the assumption of the lemma, the pair $(F,G) = (\Ind_A^B,\Res_A^B)$ is an adjoint pair of functors which are exact and preserve finite-dimensionality. Furthermore, the canonical map $FG\longrightarrow\mathrm{Id}$ is onto.
\sk

Let $M\in B\mMod$. There is a surjective map $f:FG(M)\longrightarrow M$. The kernel
of $f$ is a quotient of $FG(\ker f)$. Iterating this construction,
we obtain a complex of $B$-modules $C=\cdots\longrightarrow C^{-1}\longrightarrow C^0\longrightarrow 0$ with a morphism
$C\longrightarrow M$  that is a quasi-isomorphism and such that
$C^i$ is in $F(A\mMod)$. From Lemma \ref{le:fininlocfin} we deduce that $F(A\mMod)$ generates
$D(B\mlocfin)$ as a triangulated category closed under direct sums.

\sk

Given $i\ge 0$ and $V\in A\mMod$, we have a commutative square

\centers{$\xymatrix{
\Ext^i_{B\mMod}(F(V),M)\ar[r]^\sim \ar[d] &\Ext^i_{A\mMod}(V,G(M))\ar[d] \\
\Ext^i_{B\mMOD}(F(V),M)\ar[r]^{\sim} & \Ext^i_{A\mMOD}(V,G(M))
}$}

\noindent and we deduce that $B$ is good whenever $A$ is.  The other implication is proven in the same way, by exchanging the role of $A$ and $B$ and 
 by taking $(F,G) = (\Res_A^B,\mathrm{CoInd}_A^B)$. \end{proof}

\begin{lem}\label{le:product}
Assume $k$ is perfect and let $A,B$ be $k$-algebras. If $A$ and $B$ are good, then $A\otimes_k B$ is
good.
\end{lem}

\begin{proof}
The lemma follows from the K\"unneth Formula and the fact that 
$D^b((A\otimes_k B)\mMod)$ is generated by $A\mMod\otimes B\mMod$ as
a triangulated category, as finite dimensional simple $(A \otimes_k B)$-modules are of the form $V\otimes_k W$, where $V$ (resp. $W$) is a finite dimensional $A$-module (resp. $B$-module).
\end{proof}

\sk

\noindent \textbf{\thesubsection.2 Relative homotopy categories.\label{sect:homotopycat}}  Let $A$ be a subalgebra of a $k$-algebra $B$. We denote by
$\Ho(B,A)$ the quotient of the triangulated category $\Ho(B)$ by the thick subcategory of complexes
$C$ such that $\Res_A \, C =0$ in $\Ho(A)$. We have quotient functors
$\Ho(B)\longrightarrow \Ho(B,A) \longrightarrow D(B)$. Taking for example $A=B$ or $k$ gives
 $\Ho(B,B) = \Ho(B)$ and $\Ho(B,k) = D(B)$.

\sk

Recall that a complex $C$ of $A$-modules
is \emph{homotopically projective} if \linebreak $\Hom_{\Ho(A)}(C,D)=0$ given $D$ any acyclic complex of $A$-modules. The following lemma is classical when $A = k$ \cite[Theorem 8.1.1]{Ke}.

\begin{lem}
\label{le:homotopicallyprojective} Assume $B$ is a projective $A$-module. 
Let $\mathcal{T}$ be the full subcategory of $\Ho(B,A)$ of complexes $C$ such that
$\Res_A \, C$ is homotopically projective. 
The quotient functor induces an equivalence $\mathcal{T} \iso D(B)$, whose
inverse is a left adjoint to the quotient functor $\Ho(B,A) \longrightarrow D(B)$.
\end{lem}

\begin{proof}
Let $C$ be a homotopically projective complex of $B$-modules and $D$
be an acyclic complex of $A$-modules. Since $B$ is projective as an $A$-module, the complex $\mathrm{CoInd}_A^B \, D = \mathrm{Hom}_{A}(B,D)$ is acyclic. Consequently we
have  $\Hom_{\Ho(A)}(\Res_A^B C,D)\simeq\Hom_{\Ho(B)}(C,\mathrm{CoInd}_A^B D)=0$. It follows that $\Res_A^B C$ is homotopically projective.

\sk

Let now $C\in\mathcal{T}$. Consider a homotopically projective resolution of $C$, \ie,
a morphism of complexes $f:C'\longrightarrow C$ where $C'$ is a homotopically
projective complex and $f$ is a quasi-isomorphism.
Since $\Res_A \, C'$ and $\Res_A \, C$  are homotopically projective,
we deduce that $\Res_A\, f$ is an isomorphism in $\Ho(A)$. 
Note that an arrow $g$ of $\Ho(B,A)$ is invertible
if and only if $\Res_A \, g$ is invertible in $\Ho(A)$.
It follows that $f$ is an isomorphism in $\Ho(B,A)$.
Now, given $D\in\Comp(B)$, we have canonical isomorphisms

\centers{$\Hom_{\Ho(B)}(C',D)\iso\Hom_{\Ho(B,A)}(C',D)\iso\Hom_{D(B)}(C',D)$}

\noindent and the lemma follows.
\end{proof}

We denote by $R_D:D(B)\longrightarrow D(A)$ and $R_{\Ho}:\Ho(B,A)\longrightarrow \Ho(A)$ the
triangulated functors induced by the restriction $\Res_A^B$. 
Given $E:\mathcal{C} \longrightarrow \mathcal{C}'$ a functor and 
$\mathcal{I}\subset \mathcal{C}'$, we denote by
$E^{-1}(\mathcal{I})$ the full subcategory of $\mathcal{C}$ of objects $C$ such that
$E(C)$ is isomorphic to an object of $\mathcal{I}$.

\begin{lem}
\label{le:goodrepresentative}
Let $A$ be a finite-dimensional
subalgebra of a $k$-algebra $B$. Assume that there exists a subalgebra $B'$ of $B$
such that $B$ is a finitely generated projective $(A,B')$-bimodule. Then the quotient functor induces an equivalence \linebreak
$R_{\Ho}^{-1}(\Comp^b(A\mPROJ)) \iso R_D^{-1}(\Comp^b(A\mPROJ))$.

Furthermore, if $B$ is good, this restricts to
an equivalence  \linebreak
$R_{\Ho}^{-1}(\Comp^b(A\mproj))$ $\iso R_D^{-1}(\Comp^b(A\mproj))$.
\end{lem}

\begin{proof}
The fully faithfulness is given by Lemma \ref{le:homotopicallyprojective}.  
\sk

We construct by induction a complex of $(B,B)$-bimodules
$X=\cdots\longrightarrow X^{-1}\longrightarrow X^0\longrightarrow 0$.
We put $X^0=B\otimes_{B'}B$. Let $M$ be the kernel of the multiplication map
$X^0\longrightarrow B$. We put $X^{-1}=X^0\otimes_B M$ and $d^{-1}:X^{-1}\longrightarrow X^0$ is the
composition $X^0\otimes_B M\xrightarrow{\mathrm{mult}}M\xrightarrow{\mathrm{can}}X^0$.
Suppose $X^{-i}\longrightarrow\cdots\longrightarrow X^0\longrightarrow 0$ has been defined for some $i\ge 1$.
We put $X^{-i-1}=X^0\otimes_B\ker d^{-i}$ and
$d^{-i-1}:X^{-i-1}\xrightarrow{\mathrm{mult}}\ker d^{-i}\xrightarrow{\mathrm{can}}X^{-i}$.
The multiplication map $X\longrightarrow B$ is a quasi-isomorphism.
Note that $X^0$ is a $(B,B)$-bimodule that
is finitely generated and projective as an $(A,B)$-bimodule and as a 
$B$-module.
By induction, we deduce that $X^{-i}$
is finitely generated and projective as an $(A,B)$-bimodule and as a $B$-module, and
$\ker d^{-i}$ is a direct summand of $X^{-i}$ as a left and as a right $B$-module.
\sk

Let $C$ be an object of $R_D^{-1}(\Comp^b(A\mPROJ))$. It is a bounded complex of $B$-modules such that $\Res_A \, C$ is 
quasi-isomorphic to a bounded complex of projective modules $C'$. Let  $n\in\mathbb{Z}$ be such that the terms of $C'$ are zero in
degrees $<n$.  Let us consider the complex
$D=X\otimes_B C$. The canonical map $D \longrightarrow C$ is a quasi-isomorphism and $\Res_A \, D$ is a right bounded complex of projective modules that is quasi-isomorphic to $C'$. Consequently, $\Res_A \, D$ is homotopy equivalent to $C'$, and so is $\Res_A (\tau_{\ge n} D)$  since $\tau_{\ge n} C' =C'$. We deduce that $\tau_{\ge n} D$ is a bounded complex of $B$-modules whose restriction to $A$ are
projective. This shows the first part of the lemma. Note that if the terms of
$C$ are finite-dimensional, then the terms of
$\tau_{\ge n}D$ are finite-dimensional as well.

\smallskip
We consider finally
a bounded complex $M$ of $B$-modules whose restriction to $A$ is perfect.
Since $A$
is finite-dimensional, we deduce that $M$ has finite-dimensional total cohomology, hence it
is quasi-isomorphic to an object $C$ of $D^b(B\mMod)$, as $B$ is good. The construction
above gives a quasi-isomorphic bounded complex $\tau_{\ge n}D$ of $B$-modules whose  restriction to $A$ are
finitely generated and projective.
\end{proof}

If $B$ is a projective $A$-module we have the following picture:

\centers{$\xymatrix{
\Ho(B,A) \ar@{->>}[r] & D(B) \\
R_{\mathrm{Ho}}^{-1}(A\operatorname{\!-hoProj}) \ar[r]^\sim \ar@{^{(}->}[u] &
R_{D}^{-1}(A\operatorname{\!-hoProj})  \ar[u]^{\rotatebox{90}{$\sim$}} \\
R_{\mathrm{Ho}}^{-1}(\Comp^b(A\mPROJ)) \ \ar@{^{(}->}[r] \ar@{^{(}->}[u] &
R_{D}^{-1}(\Comp^b(A\mPROJ)) \ar@{^{(}->}[u] \\
R_{\mathrm{Ho}}^{-1}(\Comp^b(A\mproj)) \ \ar@{^{(}->}[r] \ar@{^{(}->}[u] &
R_{D}^{-1}(\Comp^b(A\mproj)) \ar@{^{(}->}[u] 
}$}

\sk

\begin{rmk} All results in \S \hyperref[sect:locfinite]{\ref{sect:locfinite}.1}-\hyperref[sect:homotopycat]{\ref{sect:homotopycat}.2} except Lemma \ref{le:product} generalise immediately to the case where the field $k$ is replaced by any commutative noetherian ring.
\end{rmk}

\sk

\noindent \textbf{\thesubsection.3 Good groups.} \label{sect:goodgroups} We relate in this section the property for a group to be good as defined in
\cite[\S 2.6, exercise 2]{Ser}, to the property
of its group algebra to be good. We refer to \cite[\S 3]{GrJZZa} for a discussion
of goodness of groups.

\sk

Let $\Upsilon$ be a group and
$\hat{\Upsilon}$ its profinite completion.
We consider only continuous representations of $\hat{\Upsilon}$, \ie, 
representations such that the orbit of any vector is finite. In particular,
we have a fully faithful embedding $k\hat{\Upsilon}\mMOD\longrightarrow k\Upsilon\mlocfin$, and this
embedding is an equivalence if $k$ is a finite field. As a consequence, we have the following result:
 
\begin{lem}\label{lem:braid}
Assume $k$ is a finite field.
The algebra $k\Upsilon$ is good if and only if given 
$M$ a finite-dimensional $k\Upsilon$-module and given $n\ge 0$,
 the canonical map
$\mathrm{H}^n(\hat{\Upsilon},M)\longrightarrow \mathrm{H}^n(\Upsilon,M)$ is bijective. 

\end{lem}

Following Serre \cite[\S 2.6, exercise 2]{Ser},
a group $\Upsilon$ is said to be
\emph{good} if for any finite $\hat{\Upsilon}$-module $M$,
the canonical map $\mathrm{H}^n(\hat{\Upsilon},M)\longrightarrow \mathrm{H}^n(\Upsilon,M)$ is an 
isomorphism for all $n$ (note that it is already bijective for $n=0,1$). It is equivalent
to the requirement that $\mathbb{F}_p\Upsilon$ is good for all primes $p$.

\sk

Let $V$ be a finite-dimensional complex vector space and let
$W$ be a finite subgroup of $\mathrm{GL}(V)$. Assume it is a complex reflection group,
\ie, it is generated by elements fixing a hyperplane. Let 
$V_{\mathrm{reg}}=\{v\in V\ |\ \mathrm{Stab}_W(v)=1\}$ and let
$x_0\in V_{\mathrm{reg}}/W$. The braid group
of $W$ is $\pi_1(V_{\mathrm{reg}}/W,x_0)$. We refer to
\cite{BMR} for basic properties of complex reflection groups and braid groups.
\sk

Recall that two groups are \emph{commensurable} if they contain isomorphic subgroups
of finite index. Lemma \ref{le:finiteindex} shows that given $\Upsilon_0$ a group
commensurable with $\Upsilon$, then $k\Upsilon$ is good if and only if
$k\Upsilon_0$ is good.
\sk

The following result generalises \cite[\S 2.6, exercise 2(d,e)]{Ser}.

\begin{prop}\label{prop:braidgood}
If $\Upsilon$ is commensurable with a free group or the braid group of a 
complex reflection group with no exceptional irreducible component
of dimension $\ge 3$, then $\Upsilon$ is good.
\end{prop}

\begin{proof}
If $\Upsilon$ is an iterative extension of free groups, then $\Upsilon$
is good by \cite[\S 2.6, exercise 2(d)]{Ser}.
\sk

Assume $\Upsilon$ is the braid group of a complex reflection
group of type $G(d,e,n)$. Then $\Upsilon$ is commensurable with an iterated extension of free
groups (cf. \cite{Na} or \cite[Remark p.152, Proposition 3.5, Lemma 3.9 and Corollary 3.32]{BMR}).
Consequently, $\Upsilon$ is good.

\sk

Finally, if $\Upsilon$ is the pure braid group of an irreducible $2$-dimensional
complex reflection group, then $Z(\Upsilon)$ is  cyclic and
$\Upsilon/Z(\Upsilon)$ is a free group \cite[p.146]{BMR},
hence $\Upsilon$ is good.
\sk

The case of the braid group of a non-irreducible complex reflection group
follows from Lemma \ref{le:product}.
\end{proof}

\begin{rmk}
We expect that the braid group of any finite complex reflection group is good.
\end{rmk}

Let $H$ be a finite normal subgroup of $\Upsilon$. It follows from
\cite[\S 2.6, exercise 2]{Ser} that $\Upsilon/H$ is good if and only if $\Upsilon$
is good and there is a finite index subgroup of
$\Upsilon$ intersecting $H$ trivially. The following proposition follows from Lemma \ref{le:goodrepresentative}.

\begin{prop}\label{prop:goodrep}
Let $k$ be a finite field.  If $\Upsilon/H$ is
good, then the quotient map induces an equivalence
from the full subcategory of $\Ho(k\Upsilon, kH)$ of complexes $C$
such that $\Res_H \, C \in\Comp^b(kH\mproj)$ to the full subcategory
of $D(k\Upsilon)$ of complexes $D$ such that $\Res_H \, D$ is perfect.
\end{prop}

\subsection{Complexes of chains with compact support\label{sect:generalvar}}

Let $k$ be a finite field of characteristic $\ell$. By variety, we will mean a quasi-projective scheme over an algebraically closed field of characteristic $p \neq \ell$. We will consider \'etale  sheaves of $k$-vector spaces. Let us recall the construction of good representatives of the complex of chains up to
homotopy. For finite group actions, the existence of such complexes is due to
Rickard \cite{Ri5}. We need here to use \cite[\S 2]{Rou}, which provides a
direct construction compatible with the action of infinite monoids.

\sk

Let $X$ by a variety acted on by a monoid $\Upsilon^+$ acting by equivalences
of the \'etale site. Let $H$ be a finite normal subgroup of $\Upsilon^+$
and $\Upsilon$ be the group deduced from $\Upsilon^+$. 
\sk

We consider   the complex of cohomology with compact support
of $X$ with value in $k$, constructed using the Godement resolution and we denote by 
$\gRgc(X)$ its $\tau_{\leq 2\dim X}$-truncation. It is viewed
as an object of $\Ho^b(k\Upsilon, kH)$. It is independent of the choice of the
compactification, up to a unique isomorphism. Most functorial properties in 
$D^b(k\Upsilon)$ lift to $\Ho^b(k\Upsilon,kH)$, in particular the triangle 
associated with an open-closed decomposition: given 
$Z$ a $\Upsilon^+$-stable closed subvariety of $X$ and $U$ the open complement,
there is a distinguished triangle in $\Ho^b(k\Upsilon,kH)$

\centers{$\gRgc(U)\longrightarrow\gRgc(X)\longrightarrow\gRgc(Z)\rightsquigarrow.$}

\begin{lem}\label{le:freeaction}
Assume the stabilisers of points in $X$ under $H$ are $\ell'$-groups. Then  $\Res_H \, \gRgc(X)$ is a bounded complex of projective modules and it is perfect.

Furthermore, if  $\Upsilon/H$ is good, then $\gRgc(X)$ is isomorphic in $\Ho^b(k\Upsilon,kH)$
to a complex $\tRgc(X)$ such that $\Res_{H} \tRgc(X) \in\Comp^b(kH\mproj)$. 

\end{lem}

\begin{proof}
It follows from \cite[\S 2.5]{Rou} that
$\Res_{H} \, \gRgc(X)$ is a bounded complex of projective $kH$-modules and from \cite[Proposition 3.5]{DeLu} that it is perfect. When $\Upsilon/H$ is good, we obtain the second part of the lemma from
Proposition \ref{prop:goodrep}.
\end{proof}

We explain now how to describe this finer invariant $\gRgc$  from the classical derived category invariant $\Rgc$ in general, by
filtrating $X$. We define a filtration of $X$ by open
$\Upsilon^+$-stable subvarieties $X_{\le i}=\{x\in X \, | \, |\Stab_H(x)|\le i\}$. Each variety $X_{\leq i-1}$ is open in $X_{\leq i}$ and the complement is a locally closed subvariety of $X$ which we will denote by $X_{i}$.
Given $Q\subset H$, we put $X_Q=\{x\in X \, | \, \Stab_H(x)= Q\}$. Given $C$ an  $ \Upsilon$-conjugacy class of subgroups of $H$, we put $X_C=\coprod_{Q\in C}X_Q$. We have a decomposition into open and closed
subvarieties $X_i=\coprod_CX_C$, where $C$ runs over the set of $ \Upsilon$-conjugacy classes of subgroups of $H$ of order $i$. Given $Q \in C$, the map $(\gamma,x) \longrightarrow\gamma x \gamma^{-1}$ induces an isomorphism $\mathrm{Ind}_{N_{ \Upsilon}(Q)}^{\Upsilon} X_Q \iso X_C$.
As a consequence, we have a distinguished triangle in $\Ho^b(k\Upsilon,kH)$

\centers{$\gRgc(X_{\le i-1})\longrightarrow \gRgc(X_{\le i})\longrightarrow
\displaystyle\bigoplus_{Q} \Ind_{N_{\Upsilon}(Q)}^{\Upsilon} \gRgc(X_{Q})\rightsquigarrow,$}

\noindent where $Q$ runs over representatives of $\Upsilon$-conjugacy classes of subgroups
of order $i$ of $H$.

\sk

The action of $N_{H}(Q)$ on $X_Q$ factors through a free action of $
N_H(Q)/Q$. 
Lemmas  \ref{le:freeaction} and \ref{le:goodrepresentative} show that $\gRgc(X_Q)$ is up to isomorphism the
unique object of $\Ho^b(kN_\Upsilon(Q), kN_H(Q))$ isomorphic in $D^b(kN_\Upsilon(Q))$ to $\Rgc(X)$ and whose restriction
to $kN_H(Q)$ is homotopy equivalent to a bounded complex of projective $(kN_H(Q)/Q)$-modules.

\mk

Recall that a $kH$-module is an \emph{$\ell$-permutation module} if it is a direct summand of a permutation module. The filtration of $\gRgc(X)$
above shows that it is isomorphic in $\Ho^b(k\Upsilon, kH)$ to a bounded complex
 of $k\Upsilon$-modules whose restrictions to $H$ are $\ell$-permutation modules. The second part of Lemma 
\ref{le:freeaction} shows the following stronger finiteness statement.

\begin{prop}\label{prop:finiteness}Assume  that $\Upsilon/H$ is good (cf. \S \hyperref[sect:goodgroups]{\ref{sect:goodgroups}.3}). Then the complex $\gRgc(X)$ is isomorphic in $\Ho^b(k\Upsilon,kH)$
to a bounded complex $\tRgc(X)$
of $k\Upsilon$-modules  whose restrictions to $H$ are finitely
generated $\ell$-permutation modules.
\end{prop}

In the setting of Brou\'e's abelian defect conjecture, we have $\Upsilon = H \rtimes B$, where $B$ is
the braid group of a complex reflection group, so that Proposition \ref{prop:finiteness} applies when
the reflection group has no exceptional component of dimension $\geq 3$ (cf. Proposition \ref{prop:braidgood}).

\sk

Let $P$ be an $\ell$-subgroup of $H$. Given $V$ an $\ell$-permutation $kH$-module, $\Br_P(V)$ is defined as the image of  the invariants $V^P$ in the coinvariants $V_P = V \otimes_{kP}k$. This construction extends to complexes of $\ell$-permutation modules. The description of
$\gRgc(X)$ above shows that the injection $X^P\hookrightarrow X$
induces an isomorphism 

\centers{$\Br_P(\gRgc(X))\iso \gRgc(X^P)$}

\noindent in $\Ho^b(kN_\Upsilon(P), kN_H(P)/P)$ (cf. also \cite[Theorem 2.29]{Rou} and \cite[Theorem 4.2]{Ri5}).

\begin{rmk}\label{rmk:ramifiedlocus} The complex $\Res_H \, \gRgc(X_1)$ is homotopy equivalent to a bounded complex of finitely
generated projective modules since by definition $H$ acts freely on $X_1$. As a consequence, the canonical map
$\gRgc(X)\longrightarrow \gRgc(X\smallsetminus X_1)$ is an isomorphism
in $\Ho^b(kH\mMOD)/\Ho^b(kH\mproj)$ (compare with Lemma \ref{lem:eigenspace}). 
\end{rmk}

\begin{rmk} Note that a finiteness property can be obtained more directly for Galois
actions.
Let $\check{C}(X,k)$ be the \v{C}ech complex of $X$, limit
of the \v{C}ech complexes $\check{C}(\mathcal{F},k)$ over the category of  \'etale coverings
$\mathcal{F}$ of $X$. The action of $\Upsilon$ on that category induces an action on
$\check{C}(X,k)$.
\sk

Assume now $X$ is endowed with a Frobenius endomorphism $F$ defining a
rational structure over a finite field. Let $\alpha\in\check{C}(X,k)$. There is
a covering $\mathcal{F}$ such that $\alpha$ is in the image of $\check{C}(\mathcal{F},k)$. The
covering $\mathcal{F}$ is isomorphic to a covering $\mathcal{F}'$ whose elements are stable under
the action of $F^n$ and such that $F^n$
acts trivially on $\check{C}(\mathcal{F}',k)$, for some $n\ge 1$.
It follows that $F^n(\alpha)=\alpha$.
So, $\check{C}(X,k)$ is locally finite for the action of $F$.
\end{rmk}

\section{Deligne-Lusztig varieties\label{sect:dlvar}}

Let $\G$ be a (not necessarily connected) reductive algebraic group, together with an isogeny $F$, some power of which is a Frobenius endomorphism. In other words, there exists a positive integer $n$ 
such that $F^n$ defines a split $\mathbb{F}_{q^n}$-structure on $\G$ for a certain power $q^n$ of the characteristic $p$, where $q \in \mathbb{R}_{>0}$. Given an $F$-stable algebraic subgroup $\mathbf{H}$ of $\G$, we will denote by $H$ the finite group of fixed points $\mathbf{H}^F$.

\sk

Let $\P = \mathbf{L} \mathbf{U}$ be a parabolic subgroup of $\G$ with unipotent radical $\U$ and an $F$-stable Levi complement $\mathbf{L}$.  We define the parabolic Deligne-Lusztig varieties

\centers{$ \begin{psmatrix}[colsep=2mm,rowsep=10mm] 
\Y_\G(\U) & = \, \big\{ g \in \G \ \big| \ g^{-1}F(g) \in F(\U)\big\}/ (\U \cap F(\U)) \\
					\X_\G(\P) & = \, \big\{ g \in \G \ \big| \ g^{-1}F(g) \in F(\P)\big\}/ (\P \cap F(\P))
\psset{arrows=->>,nodesep=3pt} 
\everypsbox{\scriptstyle} 
\ncline{1,1}{2,1}<{\pi_L}>{/ \, L}		
\end{psmatrix}$}

\sk

\noindent where $\pi_L$ denotes the restriction to $\Y_\G(\U)$ of the canonical projection $\G/(\U\cap F(\U)) \longrightarrow \G/(\P \cap F(\P))$. The varieties $\Y_\G(\U)$ and $\X_\G(\P)$  are quasi-projective varieties and endowed with a left action of $G$ by left multiplication. Furthermore, $L$ acts on the right on $\Y_{\G}(\U)$  by right multiplication and $\pi_L$ is isomorphic to the corresponding quotient map, so that it induces a $G$-equivariant isomorphism of varieties $\Y_\G(\U) / L \iso \X_\G(\P)$. 

\subsection{Fixed points and endomorphisms} 

 \noindent \textbf{\thesubsection.1 Description of fixed points.} The claim in \cite[Lemma 4.1]{Rou} can be extended to parabolic Deligne-Lusztig varieties (cf. \cite[Proposition 4.7]{DeLu} and \cite[proof of Lemma 12.3]{DM} for a related result).

\begin{lem}\label{lem:fixedpts}
Let $S$ be a finite solvable subgroup of $\Aut(\mathbf{G})$ of order prime to
$p$. Assume $S$ commutes with the action of $F$ and stabilises $\mathbf{U}$.
Then the inclusion $\mathbf{G}^S\hookrightarrow \mathbf{G}$ induces an
isomorphism

\centers{$Y_{\mathbf{G}^S}(\mathbf{U}^S)\, \iso \, Y_{\mathbf{G}}(\mathbf{U})^S.$}

\end{lem}

\begin{proof}
Denote by $\mathcal{L}_{\mathbf{G}}:\mathbf{G}\to\mathbf{G},g\longmapsto
g^{-1}F(g)$ the Lang map. We have a commutative diagram

\centers{$\xymatrix{
\mathcal{L}_{\mathbf{G}^S}^{-1} (F(\mathbf{U}^S))\ar[r]^-\sim \ar[d]&
\mathcal{L}_{\mathbf{G}}^{-1}(F(\mathbf{U}))^S \, \ar[d] \ar@{^{(}->}[r]
& \mathcal{L}_{\mathbf{G}}^{-1}(F(\mathbf{U})) \ar[d]^\alpha \\
Y_{\mathbf{G}^S}(\mathbf{U}^S)\, \ar@{^{(}->}[r] & Y_{\mathbf{G}}(\mathbf{U})^S\,
\ar@{^{(}->}[r] & Y_{\mathbf{G}}(\mathbf{U})
}$}

\noindent Assume $S$ is an $\ell$-group for some prime $\ell$. Let
$y\in Y_{\mathbf{G}}(\mathbf{U})^S$ and $V=\alpha^{-1}(y)$, an affine
space. The stratification of $V$ by stabilisers as in \S \ref{sect:generalvar} shows that
$\ell \mid \sum_i (-1)^i\dim\Hc^i(V\smallsetminus V^S,\mathbb{F}_\ell)$. We deduce that
$\Hc^*(V^S,\mathbb{F}_\ell)\not=0$, hence $V^S{\not=}\emptyset$. This proves
the lemma when $S$ is an $\ell$-group.

\sk

We prove now the lemma by induction on $|S|$. There is a non-trivial normal
$\ell$-subgroup $S_1$ of $S$ for some prime $\ell$. The canonical map
$Y_{\mathbf{G}^{S_1}}(\mathbf{U}^{S_1})\longrightarrow Y_{\mathbf{G}}(\mathbf{U})^{S_1}$
is an isomorphism. By induction, the lemma holds for $\mathbf{G}^{S_1}$ with
the action of $S/S_1$, and we deduce that the lemma holds for
$(\mathbf{G},S)$.
\end{proof}

Let $\Sigma^+$ be a monoid acting by automorphisms on $L$ and acting on the right
by equivalences of the \'etale site on  the Deligne-Lusztig variety $\Y_\G(\U)$. We assume the action is compatible with the action of $L$ and commutes with the action of $G$. We denote by $\Sigma$ the group associated with $\Sigma^+$ and we put $\Upsilon = G \times (L \rtimes \Sigma)^\opp$. 

\sk

Given $H$ a group, we denote by  $\Delta H = \{(x,x^{-1}) \, |\, x \in H\}$ the corresponding diagonal subgroup of $H \times H^\opp$. 

\begin{lem}\label{lem:goveru}Assume there exists a $\Sigma^+$-stable $p'$-subgroup $Z$ of  $L$  such that $\L = C_\G(Z)^\circ$. Then we have

\centers{$ \displaystyle \bigcup_{h \in G} h \big(\Y_\G(\U)^{\Delta Z}\big) \, = \,G / G\cap \U $} 

\noindent where $G$ acts by left multiplication and $L\rtimes \Sigma$ by right multiplication  preceded by a morphism $L\rtimes \Sigma  \longrightarrow N_G(L,G \cap \U)$ that extends the identity on $L$.

\end{lem}

\begin{proof} By assumption on $Z$, the closed subvariety $R = \bigcup_{h \in G} h \big(\Y_\G(\U)^{\Delta Z}\big)$ of $\Y_\G(\U)$ is stable by the action of  $\Upsilon^+$. Let $Q = \U^{\Delta Z} = \U \cap C_\G(Z)$. We have $\L \subset N_\G(Q)$. Since $\U \cap \L = \{1\}$ it follows that $Q$ is finite hence $\L \subset C_\G(Q)$. Since $\U \cap C_\G(\L) = \U\cap C_\P(\L) = \{1\}$ we deduce that $Q =\{1\}$. Now by Lemma \ref{lem:fixedpts} the variety $\Y_\G(\U)^{\Delta Z}$ is the image of $\mathcal{L}_{C_\G(Z)}^{-1}(F(Q)) = C_G(Z)$ by the projection $\G \longrightarrow \G / (\U\cap F(\U)$ and therefore we obtain

\centers{$ R \, = \, 
G \, (\U \cap F(\U)) / (\U \cap F(\U)) \, \simeq \, G/G\cap \U.$}

\noindent In particular the action of $L \rtimes \Sigma^+$ on $R$ induces a $G$-equivariant action on $G / G \cap \U$. 

\sk

Given $H$ a subgroup of $G$, there is a group isomorphism $N_G(H)/H \iso  \mathrm{End}_G(G/H)$ constructed as follows: an element $x H \in N_G(H)/H$ defines a $G$-equivariant map $yH \longmapsto yxH$. Conversely, the image of $H$ by a $G$-equivariant map of $G/H$ is in  $N_G(H)/H$.  Consequently the action of $\Upsilon^+$ on $R$ yields a canonical group homomorphism $L \rtimes \Sigma \longrightarrow N_G(G \cap \U)/G\cap \U$.

\sk

Let $\sigma \in \Sigma$ and $y (G\cap \U)$ be the image of $\sigma$ by  this morphism. Let $Q =$ \linebreak $ \{ (\sigma(l),l^{-1}) \, | \, l \in Z\}$. We claim that  $y (\U \cap F(\U)) \in \Y_\G(\U)^Q$. Indeed, $y^{-1} F(y) = 1$ hence $y (\U \cap F(\U)) \in  \Y_\G(\U)$. Furthermore, by definition of $y$, we have $yly^{-1} \in \sigma(l) (G\cap \U)$ and therefore $\sigma(l) y l^{-1}  \in y (G \cap \U)$ for all $l \in L$. We deduce from Lemma \ref{lem:fixedpts} that $y$ is the image of an element of $\mathcal{L}_{\G^Q}^{-1} (F(\U^Q))$, hence there exists $x \in \G^Q$ and $u \in \U \cap F(\U)$ such that $y = xu$. By definition an element $x\in \G^Q$ acts on $Z $ as $\sigma$. Consequently, $
x^{-1} F(x)$ acts by $\sigma^{-1} F \sigma$, and in particular it normalises $Z$. Now $F(y) = y$ and $u,F(u) \in F(\U)$ forces $x^{-1} F(x) \in F(\U)$. Since $N_G(Z) \subset N_G(C_\G(Z)^\circ) = N_G(\L)$ we have $F(\U) \cap N_G(Z) = \{1\}$ and we deduce that $x \in N_G(L)$ and $u \in G\cap \U$. This proves that the image of $L \rtimes \Sigma \longrightarrow N_G(G \cap \U)/G\cap \U$ lies in  $N_G(L,G \cap \U) (G\cap \U) /G\cap \U$, which is canonically isomorphic to $N_G(L,G\cap \U)$ since $\U \cap N_G(L) = \U \cap N_G(\L) = \{1\}$. 
\end{proof}

\begin{rmk} 
\begin{itemize}
\item[$\mathrm{(i)}$] There is an obstruction for equivalences on the \'etale site of $\Y_\G(\U)$ to exist: if $\sigma \in \Sigma$ acts on $L$ by conjugation by $\dot v \in N_G(L)$ then $\dot v$ will necessarily normalise $G\cap \U$. This extends the case of $F$-stable unipotent radical $\U$, for which $\Y_\G(\U) \simeq G/U$ and $N_G(L,U) = L$.

\item[$\mathrm{(ii)}$] When $G\cap \U$ is trivial, this gives no obstruction for an element of the complex reflection group $N_G(L)/L$ lifts to an equivalence on the \'etale site of $\Y_\G(\U)$. Such equivalences have already been constructed in \cite{BMi2,DM3} when $\U$ is associated with a \emph{minimal $\zeta$-element}. Note that $G\cap \U = \{1\}$ for a larger class of elements.

\item[$\mathrm{(iii)}$] The following lemma shows that one can always find a $Z$ satisfying the assumptions in Lemma \ref{lem:zexists} providing that $q$ is not too small. In the situaation of the next section, $Z$ will be a cyclic Sylow $\ell$-subgroup of $G$.

\end{itemize}
\end{rmk}

The following lemma is a variation on a classical result (cf. \cite[Lemma 13.17]{CaEn}).

\begin{lem}\label{lem:zexists}
Assume $\G$ is connected. Let $\mathbf{S}$ be an $F$-stable torus of $\mathbf{G}$ and $E$ a set of good prime numbers
for $\mathbf{G}$, distinct from $p$, and prime to
$|(Z(\mathbf{G})/Z(\mathbf{G})^\circ)^F|$. Let $Z$ be the Hall $E$-subgroup of
$S$.

If for every irreducible factor $\Phi$ of the polynomial order of
$\mathbf{S}$ there is $\ell\in E$ such that $\ell \mid \Phi(q)$, then
$C_{\mathbf{G}}(Z)^\circ=C_{\mathbf{G}}(\mathbf{S})$, and this
is a Levi subgroup of $\mathbf{G}$.
\end{lem}

\begin{proof}
Note that $C_{\mathbf{G}}(\mathbf{S})$ is a Levi subgroup by \cite[Proposition 1.22]{DM}.
\sk

Let $\mathbf{M}=C_{\mathbf{G}}(Z)^\circ$. This
is a Levi subgroup of $\mathbf{G}$ (cf. \cite[Proposition 13.16.(ii)]{CaEn}) and
$|Z|$ is prime to $|(Z(\mathbf{M})/Z(\mathbf{M})^\circ)^F|$
(cf. \cite[Proposition 13.12.(iv)]{CaEn}). Consequently, $Z\subset Z(\mathbf{M})^\circ$.
Let $\pi:\mathbf{M}\to\mathbf{M}/Z(\mathbf{M})^\circ$ be the quotient map.
By \cite[Lemma 13.17.(i)]{CaEn}, the order of $Z$ is prime to
$[\pi(\mathbf{S})^F:\pi(S)]$, hence $\pi(\mathbf{S})^F$ has order prime to $|Z|$.
On the other hand, the polynomial order of $\pi(\mathbf{S})$ divides that of
$\mathbf{S}$, hence $\pi(\mathbf{S})=1$, so $\mathbf{S}\subset Z(\mathbf{M})^\circ$
and we are done.
\end{proof}

\mk

\noindent \textbf{\thesubsection.2 Stable category and $\ell$-ramification.} Let us consider the closed $\Upsilon$-subvariety $\Y_\G(\U)_\ell$ of $\Y_\G(\mathbf{U})$ defined by 

\centers{$ \Y_\G(\U)_\ell \, = \, \big\{y \in \Y_\G(\mathbf{U}) \ \big| \ \ell \text{ divides } |\mathrm{Stab}_{G\times L^\opp}(y)|\big\}.$}

\noindent By construction, the stabilisers in $G \times L^\opp$ of points in $\Y_\G(\mathbf{U}) \smallsetminus \Y_\G(\U)_\ell$  are $\ell'$-groups and $\Y_\G(\U)_\ell$ is the smallest variety such that this property holds.

\begin{lem}\label{lem:yelldecomposition}The variety $ \Y_\G(\U)_\ell$ decomposes as 

\centers{$ \Y_\G(\U)_\ell \, = \, \displaystyle \bigcup_{\begin{subarray}{c} s \in L_\ell \setminus \{1\}\\ h \in G \end{subarray}} \Y_\G(\U)^{(h^{-1} s^{-1} h, \,s)} \, = \, \bigcup_{\begin{subarray}{c} s \in L_\ell \setminus \{1\} \\ h \in G \end{subarray}} h \big(\Y_\G(\U)\big)^{(s^{-1},\, s)} .$}

\end{lem}

\begin{proof}Let $y(\U \cap F(\U)) \in \Y_\G(\U)$ and $(g,s) \in G\times L$ be an $\ell$-element fixing $y$.  Then $g y s = yv$ for some $v \in \U \cap F(\U)$, which we can write $y^{-1}gy=vs^{-1}$. Consequently $u^{-1} v s^{-1} u = u^{-1} y^{-1} g y u = F(y^{-1} gy) = F(v)s^{-1}$. Since $v,F(v)\in F(\U)$ and $s^{-1}\in \mathbf{L} = F(\mathbf{L})$, we deduce that
$u\in C_\G(s)$. From \cite[Lemma 2.5]{DM} we deduce that $u\in F(\U)\cap C_\G(s)^\circ$. By Lang's theorem, there exists $x \in C_\G(s)^\circ$ such that $x^{-1} F(x) = u = y^{-1} F(y)$. 
With $h = x^{-1}y \in G$ we obtain $h^{-1}s^{-1}h y s = yx^{-1} s^{-1} x s = y$. \end{proof}

Given $A$ a self-injective algebra, we denote by $A\mstab$ the stable
category of $A$: it is the additive quotient $A\mstab=A\mMod/A\mproj$. The canonical
map $A\mstab\longrightarrow D^b(A\mMod)/A\mperf$, where the right-hand term is the quotient as
triangulated categories, is an equivalence of categories (Keller-Vossieck, Rickard).
This provides $A\mstab$ with a structure of  triangulated category with translation functor $\Omega^{-1}$.

\sk

From now on we assume  that $\Sigma^+$ is cyclic, generated by $\sigma$. Then the group $\Upsilon = G \times (L \rtimes \Sigma)^\opp$ is good and we have a complex $\tRgc(\Y_\G(\mathbf{U})) \in \Ho^b(k\Upsilon, k(G\times L^\opp))$ whose terms are finitely generated $\ell$-permutation $k(G\times L^\opp)$-modules  (cf. \S \ref{sect:generalvar}). 

\sk

Given $\lambda \in k^\times$ and given $M$ a finite-dimensional 
right $k\Sigma$-module, we will denote by $M_\lambda$  the generalised $\lambda$-eigenspace
of $\sigma$ (this is the image of $M\otimes_{k\Sigma}
\widehat{\vphantom{a^b}k[\sigma]}_{(\sigma-\lambda)}$ in $M$). We put
${_\lambda M}=M_{\lambda^{-1}}$, the eigenspace of $\sigma$ acting on $M$ on the left
by $\sigma^{-1}$.

\begin{lem}\label{lem:eigenspace}Given $\lambda \in k^\times$  we have an isomorphism

\centers{$ \tRgc(\X_{\G}(\P),k)_\lambda \iso  \tRgc(\Y_\G(\U)_\ell/L,k)_\lambda$}

\noindent in $kG\mstab$.
\end{lem}

\begin{proof} From Proposition \ref{prop:finiteness} we deduce that  the cone of
the canonical map $f:$ \linebreak $\tRgc(\Y_\G(\mathbf{U}))\longrightarrow \tRgc(\Y_\G(\U)_\ell)$ is homotopy equivalent to a bounded
complex of projective $k(G\times L^\opp)$-modules. As a consequence,
$\cone(f)\otimes_{kL} k$ is homotopy equivalent to a bounded complex of projective
$kG$-modules.   The map $f\otimes_{kL}k$ is a morphism of bounded complexes
of finite-dimensional $k(G\times\Sigma^\opp)$-modules, hence for any $\lambda \in k$ it induces a morphism of complexes of $\ell$-permutation $kG$-modules

\centers{$\tRgc(\X_\G(\mathbf{P}))_\lambda \longrightarrow \tRgc(\Y_\G(\U)_\ell/L)_\lambda $}

\noindent whose cone is homotopy equivalent to a bounded complex of finitely generated projective $kG$-modules.
\end{proof}

\subsection{The cyclic case}

\noindent \textbf{\thesubsection.1  Centralisers of cyclic Sylow $\ell$-subgroups.} We start by describing the centralisers of Sylow $\ell$-subgroups of $G$ under the assumption that they are cyclic.

\begin{lem}\label{lem:cyclic} Assume $\G$ is connected and $G$ has a cyclic Sylow $\ell$-subgroup $S_\ell$. Let $\L = C_\G(S_\ell)^\circ$. Then 
\begin{itemize}

\item[$\mathrm{(i)}$] $\L$ is an $F$-stable Levi subgroup of $\G$ and $S_\ell \subset Z(\L)^\circ$.

\item[$\mathrm{(ii)}$] For any non-trivial element $s \in S_\ell$,  we have $C_\G(s)^\circ  = \mathbf{L}$ and $C_G(s) = L$, hence any two distinct Sylow $\ell$-subgroups of $G$ have trivial intersection.

\item[$\mathrm{(iii)}$] $N_G(S_\ell) = N_G(L) = N_G(\L)$.

\end{itemize}
\end{lem}

\begin{proof} Let us first consider the case where $\mathbf{G}$ is simple. Let $\G_\mathrm{sc}$ be the  universal cover of $\G$. We denote by $\mathbf{O}_{\mathbf{G},F}(x)=x^N\prod_e \Phi_e(x)^{a(e)}$  the "very twisted" polynomial order of  $\mathbf{G}$: we have $|G|=\mathbf{O}_{\mathbf{G},F}(q^\varepsilon)$ where $\varepsilon =2$  if $\mathbf{G}$ has type ${^2B}_2$, ${^2F}_4$ or ${^2G}_2$, and
$\varepsilon = 1$ otherwise. Let $d$ be the order of $q^\varepsilon$ modulo $\ell$. With $S_\ell$ being cyclic, we claim that:

\begin{itemize} 
\item the multiplicity of $\Phi_d$ as a divisor of $\mathbf{O}_{\mathbf{G},F}(x)$ is $1$;
\item $\ell$ is odd;
\item $\ell \nmid |Z(\G_\mathrm{sc})^F|$. In particular, both $Z(\G)^F$ and $Z(\G^*)^F$ are $\ell'$-groups;
\item $\Phi_{d\ell^r} \nmid \mathbf{O}_{\mathbf{G},F}(x)$ for $r\ge 1$;
\item $\ell$ is good for $\mathbf{G}$.
\end{itemize}

\noindent We have $\ell{\not=2}$ by \cite[Theorem 4.10.5(a)]{GLS}. Assume now $\ell$ is odd.
If $\ell$ divides $|Z(\mathbf{G}_\mathrm{sc})^F|$, then
$W^F$ has non-cyclic Sylow $\ell$-subgroups (cf. \cite[Table 2.2]{GLS}),
unless $\mathbf{G}$ has type $A$: in that case, if $\mathbf{T}$ is a quasi-split torus of $\mathbf{G}$, then $N_{G}(\mathbf{T})$ has
non-abelian Sylow $\ell$-subgroups. We deduce that the multiplicity of $\Phi_d$ as a divisor of
$\mathbf{O}_{\mathbf{G},F}(x)$ is $1$ \cite[Theorems 4.10.2 and 4.10.3]{GLS} and
that $G_\mathrm{sc}$ has cyclic Sylow $\ell$-subgroups. The last two properties are easily checked by inspection. Note that conversely, if the multiplicity of $\Phi_d$ as a divisor of $\mathbf{O}_{\mathbf{G},F}(x)$ is $1$, then $G$ has cyclic Sylow $\ell$-subgroups \cite[Theorem 4.10.3]{GLS}. Note also that by descents of scalars, the result remains true if $\G \simeq \G_1 \times \cdots \times \G_r$ is a product of simple groups permuted cyclically by $F$ since in that case $\G^F \simeq \G_1^{F^r}$.

\sk

Now  any connected reductive group $\G$ is a product of its minimal $F$-stable semisimple normal connected subgroups and its connected center. Moreover, the intersection of any two such subgroups is finite and central, and the conditions on $\ell$ given above force $S_\ell$ to lie in only one component (since $Z(\mathbf{H})$ is a quotient of $Z(\mathbf{H}_\mathrm{sc})$ for any semisimple group $\mathbf{H}$). We may therefore assume that $(\G,F)$ is a product of simple groups permuted cyclically by $F$. 

\sk

Let $\L = C_\G(S_\ell)^\circ$ and $s \in S_\ell$. We fix a pair $(\G^*,F^*)$ dual to $(\G,F)$. By \cite[Proposition 3.16.(ii)]{CaEn}, $C_\G(s)^\circ$ is a Levi subgroup, which proves $\mathrm{(i)}$. By  \cite[Proposition 3.16.(i)]{CaEn}, the group $(C_\G(s)/C_\G(s)^\circ)^F$ is trivial since it is both an $\ell$-group and a subquotient of  $Z(\G^*)^F$ which is an $\ell'$-group (it is isomorphic to a quotient of $Z( \G_\mathrm{sc})$). This shows that $(C_\G(s)^\circ)^F = C_G(s)$.  In particular, $C_\G(s)^\circ$ contains $s$. By \cite[Proposition 13.12.(ii)]{CaEn}, its connected center $Z(C_\G(s)^\circ)^\circ$ also contains $s$. The (usual) polynomial order of $(\G,F)$ has a unique simple factor over $\mathbb{Z}[x]$ (or $\mathbb{Z}[\sqrt p][x]$ for Ree and Suzuki groups) that vanishes modulo $\ell$ at $x=q$.  Consequently, $\ell \nmid [G:(Z(C_\G(s)^\circ)^\circ)^F]$, hence $S_\ell \subset Z(C_\G(s)^\circ)^\circ$ and therefore $C_\G(s)^\circ = C_\G(S_\ell)^\circ$.

\sk

The last part of  (ii) follows from the inclusions $N_G(\L) \subset N_G(L) \subset N_G(S_\ell)  \subset N_G(Q) \subset N_G(C_G(Q)^\circ) = N_G(\L)$ given any non-trivial subgroup $Q$ of $S_\ell$.  \end{proof}

\begin{rmk}Note that $C_{\mathbf{G}}(S_\ell)$ is not always connected. For example, take
$\mathbf{G}=\mathrm{PGL}_{\ell}$ and assume $F$ defines a split structure over
$\mathbb{F}_q$. Let $d$ be the order of $q$ in $\mathbb{F}_\ell^\times$. Assume $d>1$
and $\ell^2 \nmid \Phi_d(q)$. Then, a Sylow $\ell$-subgroup
$S_\ell$ of $G$ has order $\ell$ and
$C_{\mathbf{G}}(S_\ell)/C_{\mathbf{G}}(S_\ell)^\circ$ has order $\ell$.
\end{rmk}

Let us assume now that $\G$ is a connected reductive group such that $G$ has a cyclic Sylow $\ell$-subgroup $S_\ell$. We take $\L = C_\G(S_\ell)^\circ$. It is an $F$-stable Levi subgroup of $\G$ (cf. Lemma \ref{lem:cyclic}). Given $s$ a non-trivial element of $S_\ell$, we have $C_G(s) = L$ by Lemma \ref{lem:cyclic}.(ii). We deduce from Lemmas \ref{lem:goveru} and \ref{lem:yelldecomposition} that there exists a group homomorphism $L \rtimes \Sigma \longrightarrow N_G(L,G \cap \U)$ such that
 
 \centers{$\Y_\G(\U)_\ell \simeq \, \mathrm{Res}_{G \times (L \rtimes \Sigma)^\opp}^{G \times   N_G(L,G \cap \U)^\opp} \, G / G\cap \U.$}

Let $N_\Sigma$ be the subgroup of $N_G(L)$ generated by the image of $L \rtimes \Sigma$. Let $e$ be the order of the cyclic group $N_\Sigma /L$. Given $\lambda$ an $e$-th root of unity  in $k^\times$, we denote by $k_\lambda$ the one-dimensional representation of $N_\Sigma$ on which the image of $\sigma$ acts by $\lambda$ and $N_\Sigma$ acts trivially. 

\begin{prop}\label{prop:eigenspace}Assume $\G$ is connected and $G$ has a cyclic Sylow $\ell$-subgroup $S_\ell$. Let $\L = C_\G(S_\ell)^\circ$. Given $\lambda \in k^\times$, we have

\centers{$ \tRgc(\X_\G(\P),k)_\lambda \, \simeq \, \left\{ \hskip-1.3mm \begin{array}{l} \mathrm{Ind}_{N_{\Sigma}}^G k_{\lambda} \ \ \text{if} \ \ \lambda^e =1; \\[4pt] 0 \ \ \text{otherwise} \end{array}\right.$}

\noindent in $kG\mstab$.
\end{prop}

\begin{proof} By Lemma \ref{lem:eigenspace} and the description of $\Y_\G(\U)_\ell$, we have  $\tRgc(\X_\G(\P),k)_\lambda = 0$ in $kG\mstab$ if $\lambda^e \neq 1$. Otherwise we have

\centers{$ \tRgc(\X_\G(\P),k)_\lambda \, \simeq \, \Ind_{N_\Sigma \ltimes (G\cap \U)}^G \, \Res_{N_\Sigma \ltimes (G\cap \U)}^{N_\Sigma} k_{\lambda} $}

\noindent in $kG\mstab$. Now by Lemma \ref{lem:cyclic}.(iii) we have $N_G(S_\ell) \cap \U = N_G(\L) \cap \U = \{1\}$, hence $S_\ell \cap {}^u(S_\ell) = \{1\}$ for any non-trivial $u \in G \cap \U$ (cf. Lemma \ref{lem:cyclic}.(ii)). It follows from the Mackey formula that $\Ind_{N_\Sigma }^{N_\Sigma \ltimes (G \cap \U)}  k_{\lambda}  \, \simeq \, \Res_{N_\Sigma \ltimes  (G \cap \U) }^{N_\Sigma }  k_{\lambda}$ in $k(N_\Sigma \ltimes (G \cap \U))\mstab$. \end{proof}

\begin{rmk} Note that this result  holds if we replace the condition that $S_\ell$ is a cyclic Sylow $\ell$-subgroup of $G$ by the following: $S_\ell$ is a Sylow $\ell$-subgroup of $L$ and for all non-trivial $\ell$-element $s\in L$ we have $C_\G(s)^\circ = \L$.
\end{rmk}

\mk

\noindent \textbf{\thesubsection.2 Endomorphism associated with $F$.} We can construct a specific endomorphism $\sigma$ of $\Y_\G(\U)$ associated with the Frobenius. There exists $\dot w \in N_\G(\L)$ such that ${}^{\dot w F} (\mathbf{L},\P) = (\mathbf{L},\mathbf{P})$. Let $\delta \geq 1$ be minimal such that $\dot w F$ induces a split structure on $\L$. Let us consider $\dot v = \dot w F(\dot w) \cdots F^{\delta -1}(\dot w)$ and define $\sigma = \dot v F^{\delta} = (\dot w F)^\delta$. We can choose $\dot w$ such that $\dot v$ is fixed by $F$. We let $\sigma$ act  on $\Y_\G(\U)$ by $\sigma(g) = F^\delta(g)\dot v^{-1}$. It is compatible with the action of $G \times L^\opp$, where $\sigma$ acts on $L$ by conjugation by $\dot v^{-1}$.

\begin{cor}\label{corsyz} Assume there is a cyclic Sylow $\ell$-subgroup $S_\ell$ of $G$ such that $\L = C_\G(S_\ell)^\circ$. Assume furthermore that $v = w F(w) \cdots F^{\delta -1}(w)$ generates $N_G(L)/L$. Let $m \in \{0,\ldots,e-1\}$. If 
${_{q^{m\delta}}\tRgc}(\X_\G(\P), k)$ is quasi-isomorphic to a module concentrated in degree $d$ with no projective  indecomposable summand, then there exists an isomorphism of $kG$-modules

\centers{${_{q^{m\delta}}\Hc^d}(\X_\G(\P),k) \, \simeq \,  \Omega^{2m-d} \, k.$}

\end{cor}

\begin{proof} The endomorphism $\sigma$ induces a split $\mathbb{F}_{q^\delta}$-structure on the torus $Z(\mathbf{L})^\circ$, and therefore $\dot v$ acts on $S_\ell$ by raising any element to the power of $q^{-\delta}$. In particular, since $v$ has order $e$, the image  of $q^\delta$ in $k$ is a primitive $e$-th root of unity. We deduce that the one-dimensional representation of $N_G(L)$ on which $v$ acts by $q^{m\delta}$ satisfies $k_{q^{m\delta}} \simeq \Omega^{-2m} k$ in the category of $kN_G(L)$-modules (see Example \ref{exstar} below for more details). We deduce from Proposition \ref{prop:eigenspace} that 
${_{q^{m\delta}}\tRgc}(\X_\G(\P)) \simeq \Omega^{2m} \mathrm{Ind}_{N_G(L)}^G\, k$ in the stable category of $kG$-modules.

\sk 

Since the distinct Sylow $\ell$-subgroups of $G$ have trivial intersection (Lemma \ref{lem:cyclic}.(ii)), we have $ \mathrm{Ind}_{N_G(L)}^G\, k =  \mathrm{Ind}_{N_G(S_\ell)}^G\, k\simeq k$ in the stable category. The result follows then from the fact that two $kG$-modules that have no projective indecomposable summand are isomorphic in the stable category if and only if they are isomorphic as $kG$-modules.
\end{proof}

\begin{rmk} This corollary generalises to the eigenspace of an operator $D_{\dot v}$ (as defined in \cite{BMi2,DM3}) whenever $N_G(L)$ is generated by $L$ and $\dot v$.
\end{rmk}

\section{Brauer trees}

\subsection{Walking around Brauer trees}

 Let $\ell$ be a prime number, $\mathcal{O}$ be the ring of integers of a finite extension $K$ of $\mathbb{Q}_\ell$ and let $k$ be its residue field. We will assume that $K$ is large enough for all the finite groups encountered. Let $H$ be a finite group and $b\mathcal{O} H$ be a block of $\mathcal{O} H$. If the defect $D$ of the block is a non-trivial cyclic group, then the category of $b\mathcal{O} H$-modules can be described by a combinatorial objet, the \emph{Brauer tree} $\Gamma$ of the block \cite[Chapter VII]{Fe}:
 
\begin{itemize}

\item The set of vertices $\mathscr{V}$ of $\Gamma$ consists of the ordinary non-exceptional characters in the block and the sum $\chi_{\mathrm{exc}}$ of the exceptional characters in the block. The number of non-exceptional (resp. exceptional) characters will be denoted by $e$ (resp. $m$). The integer $m$ will also be referred to as the  \emph{multiplicity} of the exceptional vertex.

\item There is an edge $\chi$ --- $\chi'$ in the Brauer tree if there exists a projective indecomposable $b\mathcal{O} H$-module with character $\chi + \chi'$ for $\chi \neq \chi' $ in $\mathscr{V}$.

\item There is a cyclic ordering of the edges containing any given vertex, defining a planar embedding of the tree.

\end{itemize}

\noindent The planar embedded Brauer tree determines the category of $b\mathcal{O} H$-modules up to Morita equivalence. 

\sk

Let us first describe the structure of the projective indecomposable modules in the block. Let $P$ be such a module, and assume that its character is the sum of two non-exceptional characters $\chi$ and $\chi'$ as in the following picture:

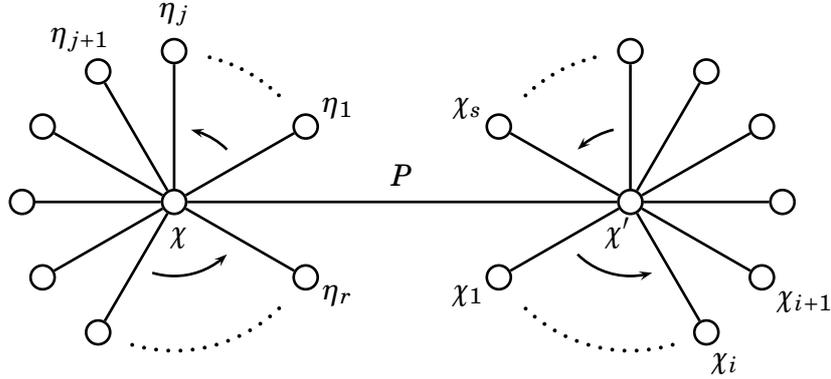
\centers[3]{\begin{pspicture}(10,4.5)
  \psset{linewidth=1pt}

    \cnode(8,2){5pt}{A}
  \cnode(10,2){5pt}{B}
  \cnode(9.73,3){5pt}{C}
  \cnode(9.73,1){5pt}{D}
    \cnode(9,3.73){5pt}{H}
    \cnode(9,0.27){5pt}{I}
    \cnode(8,4){5pt}{J}

  \cnode(6.27,3){5pt}{F}
  \cnode(6.27,1){5pt}{G}

    \cnode(2,2){5pt}{A2}
  \cnode(0,2){5pt}{B2}
  \cnode(0.27,3){5pt}{C2}
  \cnode(0.27,1){5pt}{D2}
    \cnode(1,3.73){5pt}{H2}
    \cnode(1,0.27){5pt}{I2}
    \cnode(2,4){5pt}{J2}

  \cnode(3.73,3){5pt}{F2}
  \cnode(3.73,1){5pt}{G2}

  \ncline[nodesep=0pt]{A}{B}
  \ncline[nodesep=0pt]{A}{C}
  \ncline[nodesep=0pt]{A}{D}\ncput[npos=1.5]{$\chi_{i+1}$}
  \ncline[nodesep=0pt]{A}{F}\ncput[npos=1.4]{$\chi_s$}
  \ncline[nodesep=0pt]{A}{G}\ncput[npos=1.4]{$\chi_1$}
  \ncline[nodesep=0pt]{A}{H}
  \ncline[nodesep=0pt]{A}{I}\ncput[npos=1.4]{$\chi_i$}
  \ncline[nodesep=0pt]{A}{J}
  \ncline[nodesep=0pt]{A2}{B2}
  \ncline[nodesep=0pt]{A2}{C2}
  \ncline[nodesep=0pt]{A2}{D2}
  \ncline[nodesep=0pt]{A2}{F2}\ncput[npos=1.4]{$\eta_1$}
  \ncline[nodesep=0pt]{A2}{G2}\ncput[npos=1.4]{$\eta_r$}
  \ncline[nodesep=0pt]{A2}{H2}\ncput[npos=1.4]{$\eta_{j+1}$}
  \ncline[nodesep=0pt]{A2}{I2}
  \ncline[nodesep=0pt]{A2}{J2}\ncput[npos=1.4]{$\eta_j$}\ncput[npos=-0.4]{$\ \chi$}
  \ncline[nodesep=0pt]{A}{A2}\nbput{$P$}\naput[npos=0]{$\chi'$}

\psellipticarc[linestyle=dotted,linewidth=1.5pt](8,2)(2,2){103}{135}
\psellipticarc[linestyle=dotted,linewidth=1.5pt](8,2)(2,2){225}{287}
\psellipticarc[linewidth=1pt]{->}(8,2)(1,1){103}{135}
\psellipticarc[linewidth=1pt]{->}(8,2)(1,1){225}{287}

\psellipticarc[linestyle=dotted,linewidth=1.5pt](2,2)(2,2){45}{77}
\psellipticarc[linestyle=dotted,linewidth=1.5pt](2,2)(2,2){253}{315}
\psellipticarc[linewidth=1pt]{->}(2,2)(1,1){45}{77}
\psellipticarc[linewidth=1pt]{->}(2,2)(1,1){253}{315}

\end{pspicture}}

\noindent Denote by $S_j$ (resp. $T_i$) the simple $kH$-module whose projective cover has character $\chi+\eta_j$ (resp. $\chi' + \chi_i$) over $K$ and let $\overline{P} = P \otimes_\mathcal{O} k$. Assume $\chi,\chi' \neq \chi_\mathrm{exc}$. The module $\mathrm{rad} \, \overline{P} / \mathrm{soc}\, \overline{P}$ is the direct sum of two uniserial modules with composition series $S_1, \ldots, S_r$ and $T_1,\ldots,T_s$ so that $\overline{P}$ has the following structure:
\begin{equation}\label{formproj}
 \begin{array}{ccc} & S  & \\ S_1 &  & T_1 \\ \vdots & & \vdots \\ S_r & & T_s \\ & S & \end{array}
 \end{equation}
\noindent In addition, the unique quotient $U$ (resp. submodule $V$) of $\overline{P}$ which has $S,S_1, \ldots, $ $ S_r$ (resp. $T_1, \ldots, T_s, S$) as a  composition series can be lifted  to an $\mathcal{O}$-free $\mathcal{O} H$-module of character $\chi$ (resp. $\chi'$). The structure of $\overline{P} = P_U$ yields $\Omega U \simeq V$.  Now $V$ is in turn a quotient of a projective cover of $T_1$, so that $\Omega V = \Omega^2 U$ is a uniserial module with character $\chi_1$. By iterating with process, we obtain a sequence $(\Omega^i U)_{i \geq 0}$ of uniserial modules, each of which lifts  to an $\mathcal{O}$-free $\mathcal{O} H$-module yielding an irreducible ordinary character (or the exceptional character) in the block.  This sequence is called the \emph{Green walk} starting at $U$ \cite{Gr}. It is periodic of period $2e$ and can be easily read off from the planar embedded tree. 

\begin{rmk}
When $\chi = \chi_\mathrm{exc}$, the structure of $\overline{P}$ described above is slightly different: one should turn around the exceptional node as many times as the multiplicity of the exceptional vertex. This amounts to repeating $m$ times the composition series $S, S_1, \ldots, S_r$ in $U$. 
\end{rmk}

\begin{exemple}\label{exstar}We close this section with the example of a star. Assume that $H = D \rtimes E$ where $D$ is a cyclic $\ell$-group and $E$ is an $\ell'$-subgroup of $\mathrm{Aut}(D)$. Fix a generator $x$ of $E$ of order $e$. Then $x$ acts on $D$ by raising the elements to some power $d$. By Hensel's Lemma there exists a primitive $e$-th root of unity $\zeta \in \mathcal{O}$ congruent to $d$. Denote by $\chi_1,\ldots,\chi_e$ the one-dimensional characters of $H$ over $K$ such that $\chi_i(x) = \zeta^i$ and denote by $S_1, \ldots, S_e = k$ the associated $kH$-modules. The exceptional characters are  the characters of $H$ of dimension $> 1$. The planar embedded Brauer tree of the principal $\ell$-block of $H$ is given by the following picture:

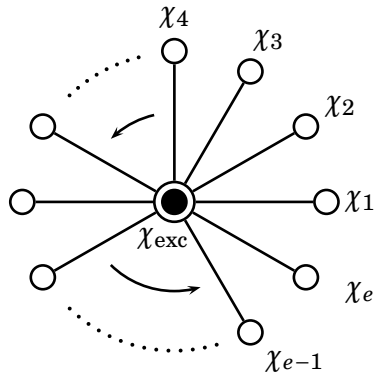
\centers[3]{\begin{pspicture}(4,4.5)

  \psset{linewidth=1pt}

   \cnode[fillstyle=solid,fillcolor=black](2,2){5pt}{A2}
       \cnode(2,2){8pt}{A}
  \cnode(4,2){5pt}{B}
  \cnode(3.73,3){5pt}{C}
  \cnode(3.73,1){5pt}{D}
    \cnode(3,3.73){5pt}{H}
    \cnode(3,0.27){5pt}{I}
    \cnode(2,4){5pt}{J}

  \cnode(0,2){5pt}{E}
  \cnode(0.27,3){5pt}{F}
  \cnode(0.27,1){5pt}{G}

  \ncline[nodesep=0pt]{A}{B}\ncput[npos=1.4]{$\chi_1$}
  \ncline[nodesep=0pt]{A}{C}\ncput[npos=1.45]{$\chi_2$}
  \ncline[nodesep=0pt]{A}{D}\ncput[npos=1.38]{$\hphantom{aaa}\chi_{e}$}
  \ncline[nodesep=0pt]{A}{E}
  \ncline[nodesep=0pt]{A}{F}
  \ncline[nodesep=0pt]{A}{G}
  \ncline[nodesep=0pt]{A}{H}\ncput[npos=1.42]{$\chi_3$}\ncput[npos=-0.55]{$\ \ \, \chi_{\mathrm{exc}}$}
  \ncline[nodesep=0pt]{A}{I}\ncput[npos=1.4]{$\hphantom{aaa}\chi_{e-1}$}
  \ncline[nodesep=0pt]{A}{J}\ncput[npos=1.4]{$\chi_4$}

\psellipticarc[linestyle=dotted,linewidth=1.5pt](2,2)(2,2){103}{135}
\psellipticarc[linestyle=dotted,linewidth=1.5pt](2,2)(2,2){225}{287}
\psellipticarc[linewidth=1pt]{->}(2,2)(1.2,1.2){103}{135}
\psellipticarc[linewidth=1pt]{->}(2,2)(1.2,1.2){225}{287}

\end{pspicture}}

\noindent In this particular case, the syzygies of a module $S_j$ satisfy $\Omega^{2i} S_j = S_{j+i}$ and  Green's walk from $S_j$ yields the sequence $\chi_j,\chi_\mathrm{exc},\chi_{j+1}, \chi_\mathrm{exc}, \chi_{j+2}, \ldots$.

\end{exemple}

\subsection{Brauer trees of the principal \texorpdfstring{$\Phi_h$}{Phih}-block \label{sect:phih}}

When $d$ is the Coxeter number, Hi\ss, L\"ubeck and Malle have formulated  in \cite{HLM} a  conjecture describing the Brauer tree of the principal $\Phi_d$-block. In this section we shall combine the results of \cite{Du4} and \S \ref{sect:dlvar} to obtain a general proof of the conjecture. This includes the determination of the previously unknown planar embedded tree for groups of type ${}^2 F_4$, $E_7$ and $E_8$. As a byproduct we obtain  a proof of the geometric version of Brou\'e's conjecture for varieties associated with Coxeter elements when $p$ is good (see Theorem \ref{thm:equiv}).

\sk

In this section, $\G$ is a connected reductive group, $\T$ is a maximal $F$-stable torus of $\G$, $\Phi = \Phi(\G,\T)$ is the corresponding root system and $W = N_\G(\T)/\T$ its Weyl group. We put $\phi = q^{-1} F $, a linear transformation  of $V = Y(\T)\otimes \mathbb{C}$.  Throughout this section we will assume that $V$ is irreducible for the action of $W \rtimes \langle \phi \rangle$. In particular, $\G$ decomposes as an almost product of simple groups that are permuted cyclically by $F$.

\mk

\noindent \textbf{\thesubsection.1 Previous results.\label{sec31}} We  assume   $\T$ is a  \emph{Coxeter torus}. This means that $\phi$ has an eigenvalue of order the Coxeter number $h$, where $h$ is the maximal possible order of an eigenvalue of $y\phi$ in $V$ for $y \in W$. In that case there exists  $w \in W$ and a $w\phi$-stable basis $\Delta$ of $\Phi$ such that each orbit of $\Phi$ under $w\phi$ contains exactly one positive root $\alpha$ such that $\phi(\alpha) <0$ (cf.  \cite[Section 7]{Sp}). Furthermore, the following properties are satisfied:

\begin{itemize}

\item the $\exp(2\pi \mathrm{i}/h)$-eigenspace of $\phi$ in $V$ is maximal and it is a line which  intersects trivially any reflecting hyperplane. As a consequence, the order of $|G|$, which is a polynomial in $q$, has $\Phi_h$ as a simple factor;

\item if $\delta$ denotes the order of $w\phi$ as an endomorphism of $V$, then $C_W(F)$ is a cyclic group of order $h_0 = h/\delta$ generated by $v = wF(w) \cdots F^{\delta-1}(w)$.
  
\end{itemize}

\noindent The basis $\Delta$ defines a $wF$-stable Borel subroup $\B$ containing $\T$. The corresponding Deligne-Lusztig variety $\X_\G(\B)$ will be referred to as a \emph{Coxeter variety} and we will denote by $r$ its dimension.

\sk

We assume the image of $q$ in $k$ is a primitive $h$-root of $1$. If $\G$ has type ${}^2 B_2$ (resp. ${}^2 F_4$, resp. ${}^2 G_2$) we assume in addition that $\ell \mid \ q^2 - q \sqrt2 +1$ (resp. $\ell \mid q^4-q^3\sqrt2 + q^2- q \sqrt2 +1$, resp. $\ell \mid q^2 -q\sqrt3+1$). The Sylow $\ell$-subgroups of $G$ are cyclic (cf. proof of Lemma \ref{lem:cyclic}) and $\T$ is the centraliser of one of them. 

\sk

Let us recall some results of Lusztig \cite{Lu} on the cohomology of  Coxeter varieties. We fix an $F$-stable lift $\dot v$ of $v$ in $N_\G(\T)$. The Frobenius endomorphism $\sigma = \dot v F^\delta$ acts (on the left) semi-simply on $\Hc^i(\X_\G(\B),K)$ and  each eigenvalue is equal  to $q^{j\delta}$ in $k$ for a unique $j \in \{0,\ldots,h_0-1\}$. The eigenspaces of $\sigma$ are mutually disjoint irreducible $KG$-modules and their characters $\{\chi_0,\ldots,\chi_{h_0-1}\}$ are the non-exceptional characters in the block. Moreover, if we fix a square root of $q^\delta$ in $K$, then each eigenvalue of $\sigma$ can be written as $\zeta q^{im\delta/2}$ for some integer $i$ and some root of unity $\zeta$ which depends only on the Harish-Chandra series of the associated eigenspace. For a given $\zeta$, the contribution of the  corresponding Harish-Chandra series to the cohomology of $\X_\G(\B)$ is given by

\centers{$ \begin{array}{c|c|@{\qquad \cdots \qquad}|c} 
\Hc^r(\X_\G(\B) ,K) & \Hc^{r+1}(\X_\G(\B) ,K)  & \Hc^{r+M_\zeta-m_\zeta}(\X_\G(\B) ,K) 
\\[4pt] \hline \chi_{m_\zeta} \vphantom{\mathop{A}\limits^n} & \chi_{m_\zeta+1} &  \chi_{M_\zeta} \\ \end{array}$}

\noindent for some $M_\zeta \ge m_\zeta$. Furthermore, according to \cite{Ge3}, the following tree

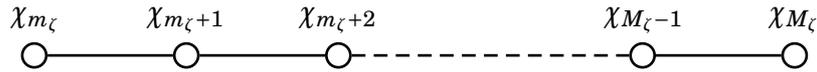
\centers{ \begin{pspicture}(10,1)
  \psset{linewidth=1pt}

  \cnode(0,0.2){5pt}{A}
  \cnode(2,0.2){5pt}{B}
  \cnode(4,0.2){5pt}{C}
  \cnode(8,0.2){5pt}{D}
  \cnode(10,0.2){5pt}{E}
  \ncline{A}{B}  \naput[npos=-0.1]{$\vphantom{\Big(}\chi_{m_\zeta}$} \naput[npos=1.1]{$\vphantom{\Big(}\chi_{m_\zeta+1}$}
  \ncline{B}{C}  \naput[npos=1.1]{$\vphantom{\Big(}\chi_{m_\zeta+2}$}
  \ncline[linestyle=dashed]{C}{D}
  \ncline{D}{E}  \naput[npos=-0.1]{$\vphantom{\Big(}\chi_{M_\zeta-1}$} \naput[npos=1.1]{$\vphantom{\Big(}\chi_{M_\zeta}$}
\end{pspicture}}

\noindent is a subtree of the Brauer tree of the principal $\ell$-block. The missing vertex is the exceptional one: it corresponds to the non-unipotent characters in the block. By \cite{Du3}, the missing edges are labelled by the cuspidal $kG$-modules in the block. With this notation, the conjecture of Hi\ss-L\"ubeck-Malle \cite{HLM} can be stated as follows:

\begin{conj}[Hi\ss-L\"ubeck-Malle] Let $\Gamma$ be the Brauer tree of the principal $\Phi_h$-block. Then 
\begin{enumerate}

\item[$\mathrm{(i)}$] (Shape of the tree) The vertices labelled by $\chi_{m_{\zeta}}$ are the only nodes connected to the exceptional node in $\Gamma$.

\item[$\mathrm{(ii)}$] (Planar embedding) The  vertices labelled by $\chi_{m_{\zeta}}$  are ordered around the exceptional vertex according to increasing values of $m_\zeta$.

\end{enumerate}
\end{conj}

Assertion (i) is known to hold for any group but $E_7$ and $E_8$. The proof relies on a case-by-case analysis, combining results of Fong-Srinivasan for  classical groups \cite{FoSr}, Hi\ss\ for  Ree  groups \cite{Hiss}, Hi\ss-L\"ubeck-Malle for groups of type $E_6$ \cite{HLM} and Hi\ss-L\"ubeck for groups of type $F_4$ and ${}^2 E_6$ \cite{HL}. The  planar embedding was not known for groups of type ${}^2 G_2$, ${}^2 F_4$, $F_4$, $E_7$ and $E_8$. The first two cases were recently settled in \cite{Du3} (under the assumption $p \neq 2,3$ in type ${}^2 F_4$), but the other cases remained unsolved. Our main result gives an unconditional proof of this conjecture:

\begin{thm}\label{thm:hlm} The conjecture of Hi\ss-L\"ubeck-Malle holds.
\end{thm}

From \cite[Theorems 3.9 and 3.11]{Du4}, we deduce that the geometric version of Brou\'e's conjecture holds for the principal $\Phi_h$-block whenever $p$ is good. Moreover, in that case the contribution to the block of the cohomology of $\Y_\G(\U)$ with coefficients in $\mathcal{O}$  is torsion-free. 

\begin{thm}\label{thm:equiv}Assume $p$ is good. Then there is a bounded complex $C$ of finitely generated $\ell$-permutation
$\mathcal{O}(G\times N_G(T)^\opp)$-modules such that
$\Res_{G\times(L\rtimes\Sigma)^\opp}^{G\times N_G(T)^\opp}(C\otimes_\mathcal{O} k)$ is
isomorphic to $\tRgc(\Y_{\mathbf{G}}(\mathbf{U}),k)$ in
$\Ho^b(k(G\times (T\rtimes\Sigma)^\opp))$
and such that $C$ induces a perverse Rickard equivalence between the principal
blocks of $\mathcal{O} G$ and $\mathcal{O} N_G(T)$.\end{thm}

\sk

\noindent \textbf{\thesubsection.2 Determination of the Brauer trees.} We shall now give a proof of the conjecture of Hi\ss-L\"ubeck-Malle. For that purpose, we shall use Corollary \ref{corsyz} and the results in \cite{Du4} to compute the syzygies of the trivial module. 
By \cite[Proposition 2.12]{Du3}, the generalised $q^{m_\zeta \delta}$-eigenspace of $\sigma = \dot v F^\delta$ on $\tRgc(\X_\G(\B),k)$ is quasi-isomorphic to a complex concentrated in  degree $r$. Moreover, it has no projective  indecomposable summand since it can be lifted up to an $\mathcal{O}$-free $\mathcal{O} G$-module of character $\chi_{m_\zeta}$. Consequently we can apply Corollary \ref{corsyz} to obtain the following isomorphism of $kG$-modules
\begin{equation}\label{eigenspace} {_{q^{m_\zeta\delta}}\Hc^{r}}(\X_\G(\B),k) \, \simeq \, \Omega^{2m_\zeta-r} \, k. 
\end{equation}
\indent The principal $\ell$-block is self-dual,  the dual of the character $\chi_{m_\zeta}$ is $\chi_{m_{\zeta^{-1}}}$ and  the dual of the module $\Omega^{2m_\zeta-r} \, k$ is $\Omega^{r-2m_\zeta} \, k$. In particular, if we apply (\ref{eigenspace}) to $\zeta$ and $\zeta^{-1}$ we deduce that both $\Omega^{2m_\zeta - r} \, k$ and $\Omega^{r-2m_{\zeta^{-1}}} \, k$ can be lifted  up to  $\mathcal{O}$-free $\mathcal{O} G$-modules with character $\chi_{m_\zeta}$. In order to compare the positions of these two modules in  Green's walk we will use the following relation:

\begin{lem} With the notation in \hyperref[sec31]{\S \ref{sec31}.1}, we have $ m_{\zeta^{-1}} + M_\zeta \, \equiv \, r \ \mathrm{mod} \ h$.
\end{lem}

\begin{proof} Recall from \cite[Table 7.1]{Lu} that the eigenvalues of $\sigma$ on the series associated with $\zeta$ are $\zeta q^{a\delta/2}, \zeta q^{\delta(a/2+1)} , \ldots, \zeta q^{\delta(a/2+M_\zeta-m_\zeta)}$ for some integer $a$. Moreover, the interval $\{a/2, \ldots, a/2 + M_\zeta -m_\zeta\}$ is centred at $r/2$, so that $ a + M_\zeta -m_\zeta \, = \,  r.$ Finally, we observe that $a$ does not change if we replace $\zeta$ by its conjugate $\zeta^{-1}$. Since $\zeta q^{a\delta/2}$ (resp. $\zeta^{-1} q^{a\delta/2}$) is equal in $k$ to $q^{\delta m_\zeta}$ (resp. $q^{\delta m_{\zeta^{-1}}})$  and $q^\delta$ has order $h$ modulo $\ell$, we deduce that $a \equiv m_\zeta + m_{\zeta^{-1}} \ \mathrm{mod} \ h$ and we conclude the proof using the previous equality.
\end{proof}

Let $1 = \zeta_1, \ldots, \zeta_s$ be the roots of unity  that  appear in the eigenvalues of $\sigma = \dot  v F^{\delta}$ on $\Hc^\bullet(\X_\G(\B),K)$, ordered according to increasing value of $m_{\zeta}$. From the lemma we  compute $ r-2m_{\zeta^{-1}} - (2m_{\zeta}-r) \, \equiv \, 2(M_{\zeta}-m_{\zeta}) \ \mathrm{mod}\ 2h$
and $2m_{\zeta_{i+1}}-r-(r-2m_{\zeta_i^{-1}}) \equiv 2 \ \mathrm{mod} \ 2h$ since $m_{\zeta_{i+1}} = M_{\zeta_i}+1$ and $m_{\zeta_1} = r$. We deduce that  Green's walk starting from the trivial module satisfies the following pattern
\begin{equation} \label{walk}
\xymatrix@C=10mm{1_G \ar[r]^{+r} & \chi_{m_{\zeta_1}} \ar[r]^{+2}   & \chi_{m_{\zeta_2}} \ar[rr]^{+2(M_{\zeta_2} - m_{\zeta_2})} & & \chi_{m_{\zeta_2}}  \ar[r]^{+2} & \chi_{m_{\zeta_3}} \ar[rr]^{+2(M_{\zeta_3} - m_{\zeta_3})} & & \cdots }
\end{equation}
\noindent The following consequence will be helpful in determining the Brauer tree:

\begin{lem}\label{lemwalk}During  Green's walk from $k$ to $\Omega^{2h} k \simeq k$, the first occurrence of a character in a Harish-Chandra series associated with $\zeta \neq 1$ is  $\chi_{m_\zeta}$.
\end{lem}

\begin{proof} If $\chi_{m_\zeta}$ is not the first character of the $\zeta$-series encountered in a Green walk, then between any two occurrences of $\chi_{m_\zeta}$, at least one character from a different series must occur. Let $\xi \neq \zeta$ be the corresponding root of unity. By the results recalled in  \hyperref[sec31]{\S \ref{sec31}.1}, every character in the block lying in this series will also occur. In particular, since the Brauer tree is a tree, any occurrence of $\chi_{m_{\xi}}$ will be found between these two occurrences of  $\chi_{m_\zeta}$, which contradicts (\ref{walk}). 
\end{proof}

 We claim that this information together with the results in \hyperref[sec31]{\S \ref{sec31}.1} is enough to determine the Brauer tree. We will only examine the case of the $\zeta_1$-series (the principal series) and the $\zeta_2$-series  as the other cases are similar. Since the distance between $\chi_{m_{\zeta_1}} = \mathrm{St}_G$ and $\chi_0 = \mathrm{1}_G$ is equal to $r$, which is also the length of the principal series,   a character is connected to the principal series in the Brauer tree if and only if it is connected to the Steinberg character:
 
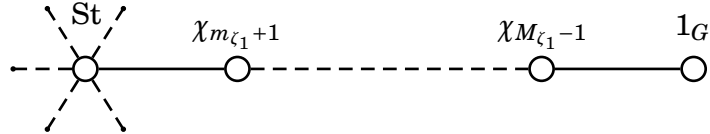
\centers{ \begin{pspicture}(12,2.5)
  \psset{linewidth=1pt}

  \cnode(2.5,2){0pt}{I2}
  \cnode(2.5,0.4){0pt}{J2}
    \cnode(1.5,2){0pt}{K2}  
  \cnode(1.5,0.4){0pt}{L2}    
   \cnode(1.05,1.2){0pt}{M2}

  \cnode(2,1.2){5pt}{B}
  \cnode(4,1.2){5pt}{C}
  \cnode(8,1.2){5pt}{D}
  \cnode(10,1.2){5pt}{E}
  \ncline{B}{C} \naput[npos=-0.1]{$\vphantom{\Bigg(}\mathrm{St}$} \naput[npos=1.1]{$\vphantom{\Big(}\chi_{m_{\zeta_1}+1}$}
  \ncline[linestyle=dashed]{C}{D}
  \ncline{D}{E}  \naput[npos=-0.1]{$\vphantom{\Big(}\chi_{M_{\zeta_1}-1}$} \naput[npos=1.1]{$\vphantom{\Big(}\mathrm{1}_G$}
    \ncline[linestyle=dashed]{B}{M2}
    \ncline[linestyle=dashed]{B}{I2}
    \ncline[linestyle=dashed]{B}{J2}
    \ncline[linestyle=dashed]{B}{K2}
    \ncline[linestyle=dashed]{B}{L2}
\end{pspicture}}

\noindent By (\ref{walk}), we know that $\chi_{m_{\zeta_2}}$ is two steps further the first occurrence of the Steinberg. If the $(r+2)$-th vertex $\chi$ in the walk is not the exceptional one then we are in the following situation

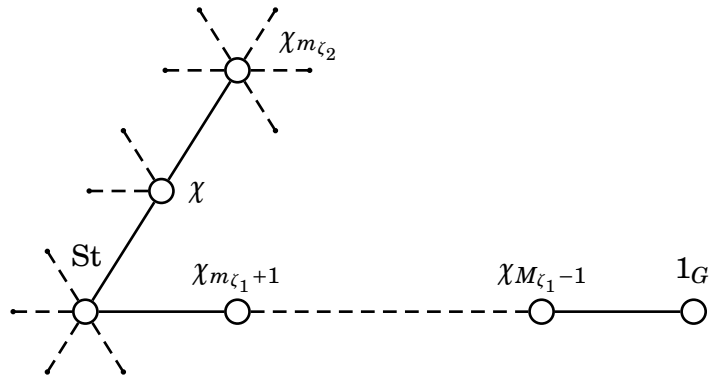
\centers{ \begin{pspicture}(12,5.5)
  \psset{linewidth=1pt}

  \cnode(3,2.8){5pt}{I2} 
  \cnode(2.05,2.8){0pt}{X2} 
    \cnode(2.5,3.6){0pt}{Y2} 
    \cnode(4,4.4){5pt}{I3}
     \cnode(3.05,4.4){0pt}{X3} 
     \cnode(4.95,4.4){0pt}{Z3} 
     \cnode(3.5,5.2){0pt}{Y3} 
    \cnode(4.5,5.2){0pt}{V3}
    \cnode(4.5,3.6){0pt}{U3} 
  \cnode(2.5,0.4){0pt}{J2}
    \cnode(1.5,2){0pt}{K2}  
  \cnode(1.5,0.4){0pt}{L2}  
     \cnode(1.05,1.2){0pt}{M2}

  \cnode(2,1.2){5pt}{B}
  \cnode(4,1.2){5pt}{C}
  \cnode(8,1.2){5pt}{D}
  \cnode(10,1.2){5pt}{E}
  \ncline{B}{C} \naput[npos=-0.1]{$\vphantom{\Bigg(}\mathrm{St}$}  \naput[npos=1.1]{$\vphantom{\Big(}\chi_{m_{\zeta_1}+1}$}
  \ncline[linestyle=dashed]{C}{D}
  \ncline{D}{E}  \naput[npos=-0.1]{$\vphantom{\Big(}\chi_{M_{\zeta_1}-1}$} \naput[npos=1.1]{$\vphantom{\Big(}\mathrm{1}_G$}
    \ncline[linestyle=dashed]{B}{M2}
    \ncline{B}{I2} \ncput[npos=1.15]{$\phantom{\Big(aaa}\chi$}
    \ncline{I2}{I3} \ncput[npos=1.45]{$\phantom{\Big(aaaaa}\chi_{m_{\zeta_2}}$}
    \ncline[linestyle=dashed]{I2}{X2}
    \ncline[linestyle=dashed]{I2}{Y2}
    \ncline[linestyle=dashed]{I3}{X3}
    \ncline[linestyle=dashed]{I3}{Y3}
    \ncline[linestyle=dashed]{I3}{Z3}
    \ncline[linestyle=dashed]{I3}{V3}
    \ncline[linestyle=dashed]{I3}{U3}     
    \ncline[linestyle=dashed]{B}{J2}
    \ncline[linestyle=dashed]{B}{K2}
    \ncline[linestyle=dashed]{B}{L2}
\end{pspicture}}

\noindent Here $\chi$ is a non-exceptional character lying in a Harish-Chandra series associated with some root of unity $\zeta$. But by parity $\chi$ cannot be equal to $\chi_{m_\zeta}$, which contradicts Lemma \ref{lemwalk}.

\sk

Therefore $\chi_{m_{\zeta_2}}$  is connected to the exceptional node and the Brauer tree has the following shape:

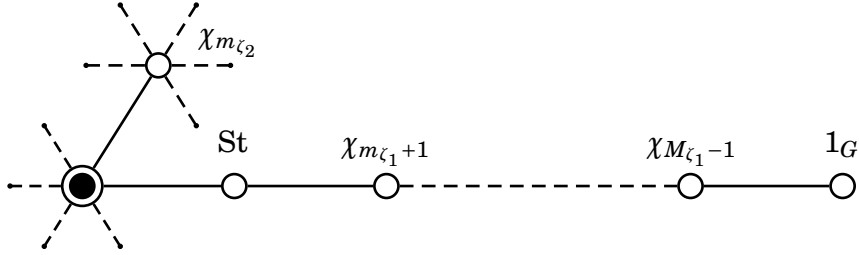
\centers{ \begin{pspicture}(10,4)
  \psset{linewidth=1pt}

  \cnode[fillstyle=solid,fillcolor=black](0,1.2){5pt}{A2}
  \cnode(0.5,0.4){0pt}{J}
    \cnode(-0.5,2){0pt}{K}  
  \cnode(-0.5,0.4){0pt}{L}    
    \cnode(-0.95,1.2){0pt}{M}    

    \cnode(1,2.8){5pt}{I3}
     \cnode(0.05,2.8){0pt}{X3} 
     \cnode(1.95,2.8){0pt}{Z3} 
     \cnode(0.5,3.6){0pt}{Y3} 
    \cnode(1.5,3.6){0pt}{V3}
    \cnode(1.5,2){0pt}{U3} 
 
  \cnode(0,1.2){8pt}{A}
  \cnode(2,1.2){5pt}{B}
  \cnode(4,1.2){5pt}{C}
  \cnode(8,1.2){5pt}{D}
  \cnode(10,1.2){5pt}{E}
  \ncline{A}{B} \naput[npos=1.1]{$\vphantom{\Big(}\mathrm{St}$}
  \ncline{B}{C}  \naput[npos=1.1]{$\vphantom{\Big(}\chi_{m_{\zeta_1}+1}$}
  \ncline[linestyle=dashed]{C}{D}
  \ncline{D}{E}  \naput[npos=-0.1]{$\vphantom{\Big(}\chi_{M_{\zeta_1}-1}$} \naput[npos=1.1]{$\vphantom{\Big(}\mathrm{1}_G$}
    \ncline[linestyle=dashed]{A}{J}
    \ncline[linestyle=dashed]{A}{K}
    \ncline[linestyle=dashed]{A}{L}
    \ncline[linestyle=dashed]{A}{M}
    \ncline{A}{I3} \ncput[npos=1.45]{$\phantom{\Big(aaaaa}\chi_{m_{\zeta_2}}$}
    \ncline[linestyle=dashed]{I3}{X3}
    \ncline[linestyle=dashed]{I3}{Y3}
    \ncline[linestyle=dashed]{I3}{Z3}
    \ncline[linestyle=dashed]{I3}{V3}
    \ncline[linestyle=dashed]{I3}{U3}     

\end{pspicture}}

\noindent Finally, since the branch corresponding to the $\zeta_2$-series has $M_{\zeta_2}-m_{\zeta_2}$ edges, we deduce from (\ref{walk}) that the Brauer tree has the following shape

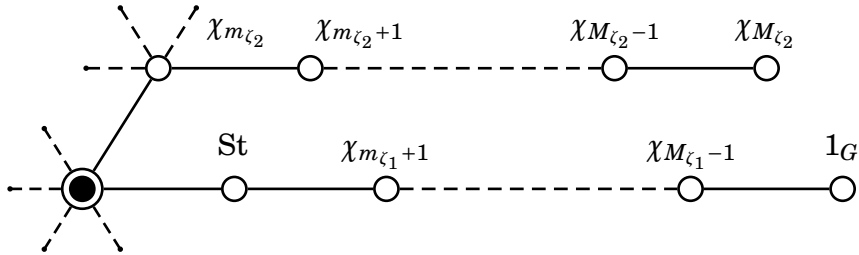
\centers{ \begin{pspicture}(11,4)
  \psset{linewidth=1pt}

  \cnode[fillstyle=solid,fillcolor=black](0,1.2){5pt}{A2}
  \cnode(0.5,0.4){0pt}{J}
    \cnode(-0.5,2){0pt}{K}  
  \cnode(-0.5,0.4){0pt}{L}    
    \cnode(-0.95,1.2){0pt}{M}    

    \cnode(1,2.8){5pt}{I3}
     \cnode(0.05,2.8){0pt}{X3} 
     \cnode(3,2.8){5pt}{Z3} 
     \cnode(0.5,3.6){0pt}{Y3} 
    \cnode(1.5,3.6){0pt}{V3}
  \cnode(7,2.8){5pt}{D3}
  \cnode(9,2.8){5pt}{E3}
 
  \cnode(0,1.2){8pt}{A}
  \cnode(2,1.2){5pt}{B}
  \cnode(4,1.2){5pt}{C}
  \cnode(8,1.2){5pt}{D}
  \cnode(10,1.2){5pt}{E}
  \ncline{A}{B} \naput[npos=1.1]{$\vphantom{\Big(}\mathrm{St}$}
  \ncline{B}{C}  \naput[npos=1.1]{$\vphantom{\Big(}\chi_{m_{\zeta_1}+1}$}
  \ncline[linestyle=dashed]{C}{D}
  \ncline{D}{E}  \naput[npos=-0.1]{$\vphantom{\Big(}\chi_{M_{\zeta_1}-1}$} \naput[npos=1.1]{$\vphantom{\Big(}\mathrm{1}_G$}
\ncline{I3}{Z3}  \naput[npos=1.5]{$\vphantom{\Big(}\chi_{m_{\zeta_2}+1}$}
  \ncline[linestyle=dashed]{Z3}{D3}
  \ncline{D3}{E3}  \naput[npos=-0.1]{$\vphantom{\Big(}\chi_{M_{\zeta_2}-1}$} \naput[npos=1.1]{$\vphantom{\Big(}\chi_{M_{\zeta_2}}$}
    \ncline[linestyle=dashed]{A}{J}
    \ncline[linestyle=dashed]{A}{K}
    \ncline[linestyle=dashed]{A}{L}
    \ncline[linestyle=dashed]{A}{M}
    \ncline{A}{I3} \ncput[npos=1.6]{$\phantom{\Big(aaaaa}\chi_{m_{\zeta_2}}$}
    \ncline[linestyle=dashed]{I3}{X3}
    \ncline[linestyle=dashed]{I3}{Y3}
    \ncline[linestyle=dashed]{I3}{V3}

\end{pspicture}}

\noindent It remains to iterate the process to obtain the planar embedded Brauer tree predicted by the conjecture of Hi\ss-L\"ubeck-Malle. This completes the proof of Theorem \ref{thm:hlm}.

\subsection{Non-unipotent  \texorpdfstring{$\ell$}{l}-blocks}

We keep the assumption on $\ell$ given in \S \ref{sect:phih} and we fix a pair $(\G^*,\T^*, F^*)$  dual to $(\G,\T, F)$.  Throughout this section, we will assume that $\G$ is adjoint, so that the centraliser of any semisimple element in $\G^*$ is connected. Let $s \in G^*$ be a semisimple $\ell'$-element. Following \cite{BMi1}, we denote by $\mathcal{E}_\ell(G,(s)) $ the union of  rational series $\mathcal{E}(G,(st))$ where $t$ runs over a set of representatives of conjugacy classes of semisimple $\ell$-elements in $C_{G^*}(s)$. By \cite[Th\'eor\`eme 2.2]{BMi1} and \cite[\S 3.2]{CaEn2}, it is a union of blocks that have either trivial or full defect. Furthermore, if $\mathcal{E}_\ell(G,(s)) $ contains a non-trivial block, then $s$ must be conjugate to an element of $C_{G^*}(S_\ell^*) = T^*$, where $S_\ell^*$ denotes the Sylow $\ell$-subgroup of $T^* = (\T^{*})^{F^*}$. Note that by \cite[Th\'eor\`eme 3.2]{BMi1}, the principal block is the only unipotent block with non-trivial defect.
 
 \sk

We assume now that $s$ is an $\ell'$-element of $T^*$ such that $\mathcal{E}_\ell(G,(s))$ contains a non-trivial block $b$. If there is a proper $F^*$-stable Levi subgroup $\L^*$ of $\G^*$ such that $C_{\G^*}(s) \subset \L^*$, then by \cite[Th\'eor\`eme B']{BR1} the block $b$ is Morita equivalent to an $\ell$-block $b_s$ in $\mathcal{E}_\ell(L,(s))$. Now $[\L^*,\L^*]$ is adjoint and two cases can arise: if $\ell \nmid |Z(\L^*)^F|$ then $[\L^*,\L^*]$ has a minimal $F$-stable connected normal subgroup that contains $s$. It is also the unique one whose Coxeter number is $h$. If $S_\ell^*$ is central in $L^*$ then $b_s$ is isomorphic to a unipotent block of $L$ and the results in \S \ref{sect:phih} apply and the Brauer tree of the block is known by induction on the rank of $\G$.

\sk

If such a proper Levi subgroup $\L^*$ does not exist then  $s$ and by extension $b$ are said to be \emph{isolated}.  Under our assumptions on $\ell$, very few isolated blocks with non-trivial defect can appear in exceptional adjoint groups:

\begin{lem} Let $(\G,F)$ be an exceptional adjoint simple group not of type $G_2$. Under the assumptions on $\ell$ in \S \ref{sect:phih}, any isolated $\ell$-block has trivial defect.
\end{lem}

\begin{proof} Let $\mathbf{M}^* = C_\G^*(s)$. It is a connected reductive subgroup of $\G^*$. If $\ell$ does not divide $[\mathbf{M}^*,\mathbf{M}^*]^{F^*}$,  then $S_\ell^*$ must be central in $\mathbf{M}^*$  since $\mathbf{M}^* = [\mathbf{M}^*,\mathbf{M}^*] \, Z(\mathbf{M}^*)$. Therefore $\mathbf{M}^* \subset  C_{\G^*}(S_\ell^*) = \T^*$. Consequently  $\ell$ divides the order of $[\mathbf{M}^*,\mathbf{M}^*]^{F^*}$ whenever $s$ is isolated.

\sk

Recall that the conjugacy classes of isolated elements are parametrised by roots in the extended Dynkin diagram (given for example in \cite{Bon2}). With the assumption on \S \ref{sect:phih}, there is a unique (tp)-cyclotomic polynomial $\Psi$ such that $\ell$ divides $\Psi(q)$ and $\Psi$ divides the polynomial order of $G$ (we have $\Psi = \Phi_h$ if we exclude the Ree and Suzuki groups). Therefore the group $[\mathbf{M}^*,\mathbf{M}^*]^{F^*}$ contains a non-trivial $\ell$-subgroup if and only if $\Psi$ appears in its polynomial order, that is if $[\mathbf{M}^*,\mathbf{M}^*]$ contains an $F^*$-stable component with the same Coxeter number as $(\G,F)$.  The Coxeter numbers for exceptional groups are given in Table \ref{table:coxeter}.

\begin{table}[h!]

\centers{$ \begin{array}{c|c|c|c|c|c|c|c|c|c|c}
(\G,F) & \vphantom{\Big(} {}^2 B_2 &  {}^3 D_4 & {}^2 E_6 & E_6 & E_7 & E_8 & {}^2 F_4 & F_4 & {}^2 G_2 & G_2 \\\hline
h & \vphantom{\Big(} 8 & 12 & 18 & 12 & 18 & 30 & 24 & 12 & 12 & 6 \\
 \end{array}$}

\caption{Coxeter numbers for exceptional groups}
 \label{table:coxeter}
\end{table}

\noindent Using the extended Dynkin diagram,  one can check that  the only centralisers of isolated elements that have the same Coxeter number are $A_2 \times A_2 \times A_2$ realised as ${}^2 A_2 (q^3)$ for ${}^2 E_6(q)$ and as $A_2(q) \times {}^2 A_2 (q^2)$ for $E_6$, and $A_2$ realised as ${}^2 A_2(q)$ for $G_2$. By \cite[\S 2]{DeLi} and \cite[\S 2]{FlJa}, the first two cases never happen for simply connected groups. 
\end{proof}

For classical groups, the $\ell$-blocks with cyclic defect have been determined in \cite{FoSr}. For $G_2$ they are given in \cite{Sha} (note that when $q\equiv -1$ modulo 3 there exists a non-trivial quasi-isolated block). As a consequence, the Jordan decomposition provides an inductive argument for determining all the $\ell$-blocks up to Morita equivalence.

\begin{thm} Assume  $\G$ is an adjoint simple group. In the Coxeter case, the Brauer tree of any non-trivial $\ell$-block of $G$ is known.
\end{thm}

\begin{rmk} Using \cite{DeLi} and \cite{Der}, one can check that for groups of type ${}^2 B_2$, ${}^3 D_4$, ${}^2 E_6$, $E_8$, ${}^2 F_4$ and ${}^2 G_2$, the order of the derived group of the centraliser of any semisimple element is coprime to $\ell$. Therefore any non-principal $\ell$-block will be either trivial or Morita equivalent to $\mathcal{O}S_\ell$. 
\end{rmk}

\subsection{New planar embedded Brauer trees}

We give here the new Brauer trees that we have obtained. Note that the shape of the trees for ${}^2 F_4$ and $F_4$ were already known by \cite{Hiss} and \cite{HL} but the planar embeddings was known for $F_4$ and $p \neq 2,3$ only (cf. \cite{Du4}). We have used the package \textsf{CHEVIE} of \textsf{GAP3} to label the irreducible unipotent characters with the convention that $1$, $\varepsilon$ and $r$ stand respectively for the trivial, the sign and the reflection representation of a Coxeter group.

\mk

\noindent \textbf{\thesubsection.1 Type ${}^2 F_4$.} Here $q = 2^m\sqrt2$ for some integer $m \geq 1$. The Coxeter case corresponds to prime numbers $\ell$ dividing $\Phi_{24}'(q) = q^4-\sqrt 2 q^3+q^2-\sqrt 2q+1$. The class of $q$ in $k$ is a primitive $24$-th root of unity. We have denoted by $\theta$ (resp. $\mathrm{i}$, resp. $\eta$) the unique primitive $3$-rd  (resp. $4$-th, resp. $8$-th) root of unity which is equal to $q^8$ (resp. $q^6$, resp. $q^3$) in $k$. The planar embedded Brauer tree of the principal $\ell$-block is given by Fig. \ref{2F4}.

\begin{figure}[h!] \hskip 5cm \begin{pspicture}(6,4.5)

  \cnode[fillstyle=solid,fillcolor=black](1.5,2){5pt}{A2}
    \cnode(1.5,2){8pt}{A}
  \cnode(3,2){5pt}{B}
  \cnode(4.5,2){5pt}{C}
  \cnode(6,2){5pt}{D}
  \cnode(0,2){5pt}{O}
  \cnode(2.56,3.06){5pt}{P}
  \cnode(1.5,3.5){5pt}{Q}
  \cnode(0.44,3.06){5pt}{R}  \cnode(-.62,4.12){5pt}{R2}
  \cnode(2.56,0.94){5pt}{S}
  \cnode(1.5,0.5){5pt}{T}
  \cnode(0.44,0.94){5pt}{U}\cnode(-.62,-0.12){5pt}{U2}

  \ncline[nodesep=0pt]{A}{B}\naput[npos=1.1]{$\vphantom{\big(_p} \mathrm{St}$}
  \ncline[nodesep=0pt]{B}{C}\naput[npos=1.1]{$\vphantom{\Big(}\phi_{2,1}$}
  \ncline[nodesep=0pt]{C}{D}\naput[npos=1.15]{$\vphantom{\big(_p}  1$}
  \ncline[nodesep=0pt]{O}{A}\ncput[npos=-1.2]{$\vphantom{\Big(} {}^2F_{4}^{\mathrm{II}}[-1] $}
 \ncline[nodesep=0pt]{A}{P} \ncput[npos=1.8]{${}^2 F_4[-\theta^2]$}
 \ncline[nodesep=0pt]{A}{Q}\ncput[npos=1.7]{${}^2 F_4^{\mathrm{II}}[\mathrm{i}]$}
 \ncline[nodesep=0pt]{A}{R}\ncput[npos=1.2]{${}^2 B_2[\eta^3]_\varepsilon$ \hphantom{aaaiaaaaaa}} 
  \ncline[nodesep=0pt]{R}{R2}\ncput[npos=1.2]{${}^2 B_2[\eta^3]_1$ \hphantom{aaaaaaaaa}} 

 \ncline[nodesep=0pt]{A}{S}\ncput[npos=1.8]{${}^2 F_4[-\theta]$}
 \ncline[nodesep=0pt]{A}{T}\ncput[npos=1.7]{${}^2 F_4^{\mathrm{II}}[-\mathrm{i}]$}
 \ncline[nodesep=0pt]{A}{U}\ncput[npos=1.2]{${}^2 B_2[\eta^5]_\varepsilon$  \hphantom{aaaaaaaaaa}} 
  \ncline[nodesep=0pt]{U}{U2}\ncput[npos=1.2]{${}^2 B_2[\eta^5]_1$  \hphantom{aaaaiaaaa}}

\end{pspicture}

\vskip 0.7cm

\caption{Brauer tree of the principal $\Phi_{24}'$-block of ${}^2F_4$}
\label{2F4}

\end{figure}
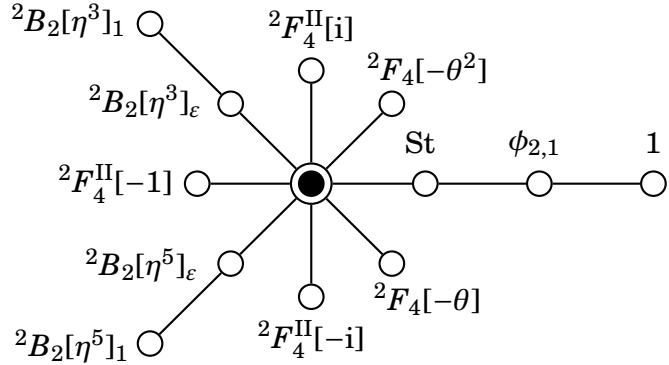

\mk

\noindent \textbf{\thesubsection.2 Type $F_4$.} The Coxeter case corresponds to prime numbers $\ell$ dividing $\Phi_{12}(q) = q^4-q^2 +1$. The class of $q$ in $k$ is a primitive $12$-th root of unity. We have denoted by $\theta$ (resp. $\mathrm{i}$) the unique primitive $3$-rd  (resp. $4$-th) root of unity which is equal to $q^4$ (resp. $q^3$) in $k$. The planar embedded Brauer tree of the principal $\ell$-block is given by Fig. \ref{F4}.

\begin{figure}[h!] \hskip 1cm \begin{pspicture}(12,4)

  \cnode[fillstyle=solid,fillcolor=black](4.5,2){5pt}{A2}
    \cnode(4.5,2){8pt}{A}
  \cnode(6,2){5pt}{B}
  \cnode(7.5,2){5pt}{C}
  \cnode(9,2){5pt}{D}
  \cnode(10.5,2){5pt}{E}
  \cnode(12,2){5pt}{F}  
  \cnode(0,2){5pt}{M}
  \cnode(1.3,2){5pt}{N}
  \cnode(3,2){5pt}{O}
  \cnode(5.25,3.3){5pt}{P}
  \cnode(3.75,3.3){5pt}{Q}  
  \cnode(5.25,0.7){5pt}{R}
  \cnode(3.75,0.7){5pt}{S}

  \ncline[nodesep=0pt]{A}{B}\naput[npos=1.1]{$\vphantom{\big(_p} \mathrm{St}$}
  \ncline[nodesep=0pt]{B}{C}\naput[npos=1.1]{$\vphantom{\Big(}\phi_{4,13}$}
  \ncline[nodesep=0pt]{C}{D}\naput[npos=1.1]{$\vphantom{\Big(}  \phi_{6,6}''$}
  \ncline[nodesep=0pt]{D}{E}\naput[npos=1.1]{$\vphantom{\Big(}\phi_{4,1}$}
  \ncline[nodesep=0pt]{E}{F}\naput[npos=1.15]{$\vphantom{\big(_p} 1$}
  \ncline[nodesep=0pt]{M}{N}\naput[npos=-0.1]{$\vphantom{\Big(}B_{2,1} $}
  \ncline[nodesep=0pt]{N}{O}\naput[npos=-0.1]{$\vphantom{\Big(}B_{2,r} $}
  \ncline[nodesep=0pt]{O}{A}\naput[npos=-0.1]{$\vphantom{\Big(} B_{2,\varepsilon} $}
\ncline[nodesep=0pt]{A}{P} \ncput[npos=1.7]{$F_4[\mathrm{i}]$}
\ncline[nodesep=0pt]{A}{Q}\ncput[npos=1.7]{$F_4[\theta]$}
\ncline[nodesep=0pt]{A}{R}\ncput[npos=1.7]{$F_4[-\mathrm{i}]$}
\ncline[nodesep=0pt]{A}{S}\ncput[npos=1.7]{$F_4[\theta^2]$}

\end{pspicture}

\vskip 0.2cm

\caption{Brauer tree of the principal $\Phi_{12}$-block of $F_4$}
\label{F4}

\end{figure}
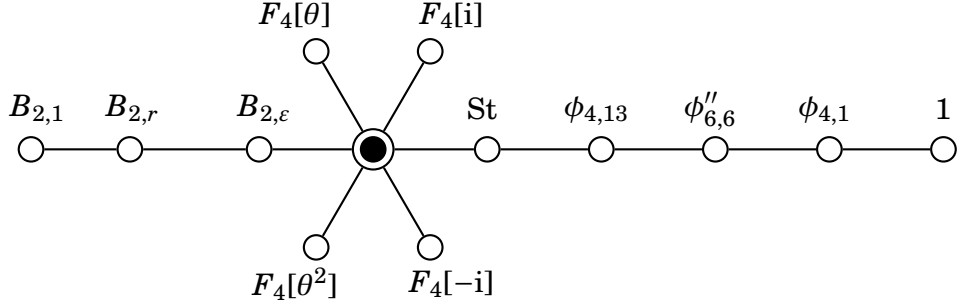

\mk

\noindent \textbf{\thesubsection.3 Type $E_7$.} The Coxeter case corresponds to prime numbers $\ell$ dividing $\Phi_{18}(q) = q^6-q^3 +1$. The class of $q$ in $k$ is a primitive $18$-th root of unity. We fix a square root $\sqrt q \in \mathcal{O}$ of $q$. We have denoted by $\theta$ (resp. $\mathrm{i}$) the unique primitive $3$-rd  (resp. $4$-th) root of unity which is equal to $q^6$ (resp. $(\sqrt q)^{9}$) in $k$. The planar embedded Brauer tree of the principal $\ell$-block is given by Fig. \ref{E7}.

\mk

\noindent \textbf{\thesubsection.4 Type $E_8$.} The Coxeter case corresponds to prime numbers $\ell$ dividing $\Phi_{30}(q) = q^8 + q^7 - q^5-q^4-q^3 +q  +1$. The class of $q$ in $k$ is a primitive $30$-th root of unity. We fix a square root $\sqrt q \in \mathcal{O}$ of $q$. We have denoted by $\theta$ (resp. $\mathrm{i}$, resp. $\zeta$) the unique primitive $3$-rd  (resp. $4$-th, resp. $5$-th) root of unity which is equal to $q^{10}$ (resp. $(\sqrt q)^{15}$, resp. $q^6$) in $k$. The planar embedded Brauer tree of the principal $\ell$-block is given by Fig. \ref{E8}.

\begin{rmk}\emph{(Communicated by David Craven)} From the Coxeter case in $E_7$ one can also deduce the Brauer trees of the principal $\Phi_{18}$-block of $E_8$ and its Alvis-Curtis dual. Assume $q$ has order $18$ modulo $\ell$. Then there exists an $F$-stable parabolic subgroup $\P = \L \U$ of $\G$ with $F$-stable Levi complement $\L$ such that  $\ell \nmid [G:L]$ (take $\L$ to be the centraliser of a $\Phi_{18}$-torus).  Let $c$ (resp. $b$) be the principal $\ell$-block of $L$ (resp. $G$). Then $b(\mathcal{O} G/U) c$ is a finitely generated $(b\mathcal{O}G, c\mathcal{O}L)$-bimodule that is projective as a $b\mathcal{O}G$-module and as a right $c\mathcal{O}L$-module. Moreover one can check that the functor $b (K G/U) c \otimes_{KL} -$ induces a bijection between the irreducible characters in $c$ and $b$. By \cite[Th\'eor\`eme 0.2]{Bro1}, we deduce that the functor $b \mathcal{O} G/U c \otimes_{\mathcal{O}L} -$ induces a Morita equivalence between $c\mathcal{O}L$ and $b\mathcal{O}G$. In particular we obtain the planar embedded Brauer tree of the principal $\Phi_{18}$-block of $E_8$ from the tree of the principal $\Phi_{18}$-block of $E_7$. The same argument applies to the $\ell$-block of $G$ containing the Steinberg character.
\end{rmk}

\begin{landscape}

\begin{figure}[h!] \hskip 1cm \begin{pspicture}(18,9)

  \cnode[fillstyle=solid,fillcolor=black](6,2){5pt}{A2}
    \cnode(6,2){8pt}{A}
  \cnode(7.5,2){5pt}{B}
  \cnode(9,2){5pt}{C}
  \cnode(10.5,2){5pt}{D}
  \cnode(12,2){5pt}{E}
  \cnode(13.5,2){5pt}{F}
  \cnode(15,2){5pt}{G}
  \cnode(16.5,2){5pt}{H}
  \cnode(18,2){5pt}{K}  
  \cnode(0,2){5pt}{L}
  \cnode(1.5,2){5pt}{M}
  \cnode(3,2){5pt}{N}
  \cnode(4.5,2){5pt}{O}
  \cnode(6.75,3.3){5pt}{P}
  \cnode(5.25,3.3){5pt}{Q}  \cnode(4.5,4.6){5pt}{Q2}
  \cnode(6.75,0.7){5pt}{R}
  \cnode(5.25,0.7){5pt}{S} \cnode(4.5,-0.6){5pt}{S2}

  \ncline[nodesep=0pt]{A}{B}\naput[npos=1.1]{$\vphantom{\big(_p} \mathrm{St}$}
  \ncline[nodesep=0pt]{B}{C}\naput[npos=1.1]{$\vphantom{\Big(}\phi_{7,46}$}
  \ncline[nodesep=0pt]{C}{D}\naput[npos=1.1]{$\vphantom{\Big(}  \phi_{21,33}$}
  \ncline[nodesep=0pt]{D}{E}\naput[npos=1.1]{$\vphantom{\Big(}\phi_{35,22}$}
  \ncline[nodesep=0pt]{E}{F}\naput[npos=1.1]{$\vphantom{\Big(} \phi_{35,13}$}
  \ncline[nodesep=0pt]{F}{G}\naput[npos=1.1]{$\vphantom{\Big(} \phi_{21,6}$}
  \ncline[nodesep=0pt]{G}{H}\naput[npos=1.1]{$\vphantom{\Big(} \phi_{7,1}$}
  \ncline[nodesep=0pt]{H}{K}\naput[npos=1.15]{$\vphantom{\big(_p} 1 $}
  \ncline[nodesep=0pt]{L}{M}\naput[npos=-0.1]{$\vphantom{\Big(} D_{4,1}$}
  \ncline[nodesep=0pt]{M}{N}\naput[npos=-0.1]{$\vphantom{\Big(}D_{4,r} $}
  \ncline[nodesep=0pt]{N}{O}\naput[npos=-0.1]{$\vphantom{\Big(}D_{4,r\varepsilon} $}
  \ncline[nodesep=0pt]{O}{A}\naput[npos=-0.1]{$\vphantom{\Big(} D_{4,\varepsilon} $}
\ncline[nodesep=0pt]{A}{P} \ncput[npos=1.7]{$E_7[\mathrm{i}]$}
\ncline[nodesep=0pt]{A}{Q}\naput[npos=1.5]{$E_6[\theta]_\varepsilon$}
\ncline[nodesep=0pt]{Q}{Q2}\naput[npos=1.5]{$E_6[\theta]_1$}
\ncline[nodesep=0pt]{A}{R}\ncput[npos=1.7]{$E_7[-\mathrm{i}]$}
\ncline[nodesep=0pt]{A}{S}\nbput[npos=1.5]{$E_6[\theta^2]_\varepsilon$}
\ncline[nodesep=0pt]{S}{S2}\nbput[npos=1.5]{$E_6[\theta^2]_1$}

\end{pspicture}

\vskip 1.5cm

\caption{Brauer tree of the principal $\Phi_{18}$-block of $E_7$}
\label{E7}

\end{figure}
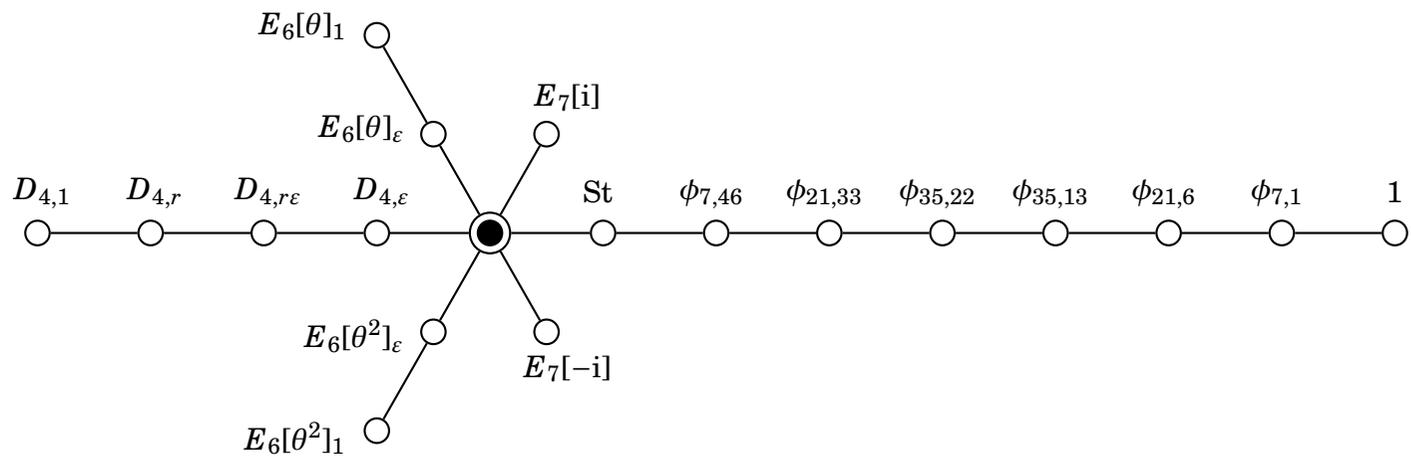

\end{landscape}

\begin{landscape}

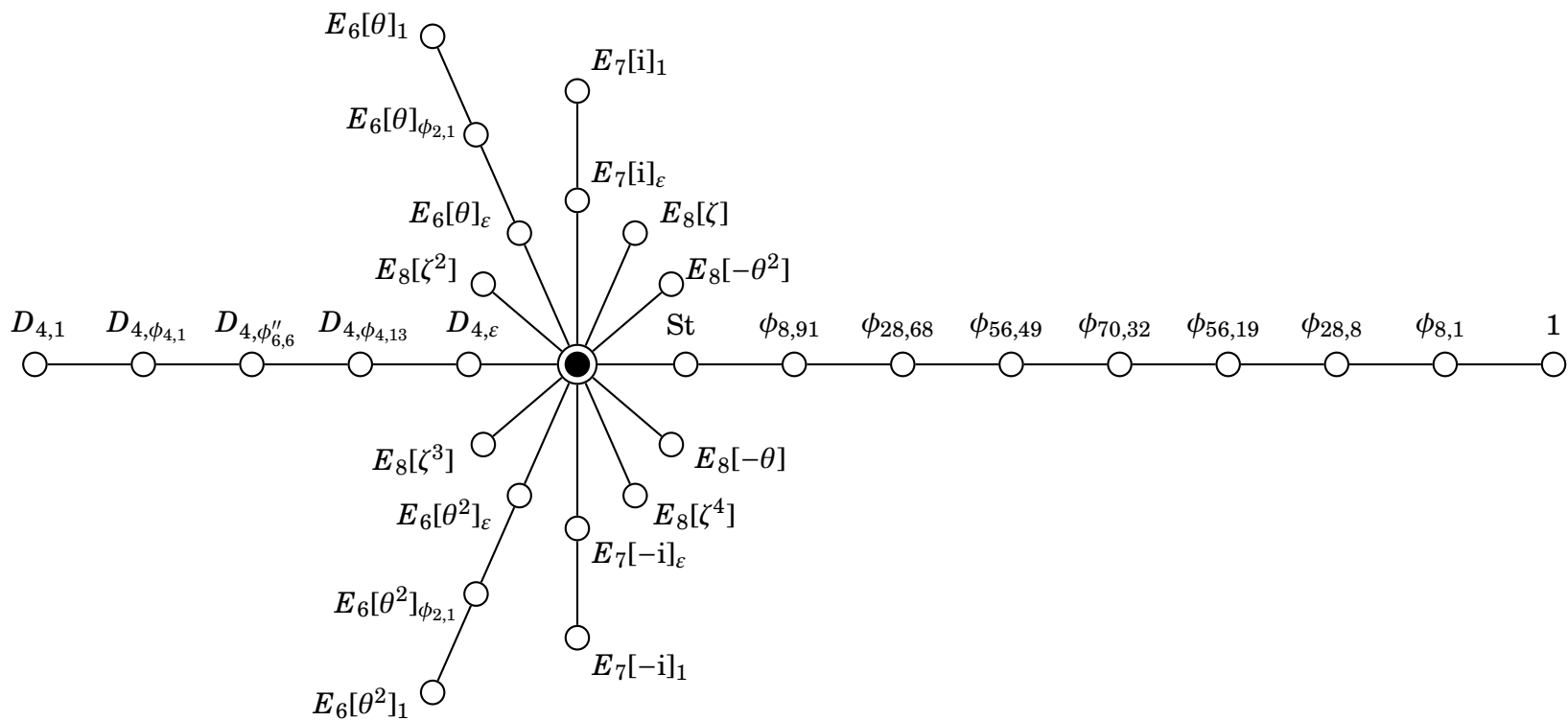
\begin{figure}[h!] \hskip 1cm\begin{pspicture}(19.5,8)

  \cnode[fillstyle=solid,fillcolor=black](6,2){5pt}{A2}
  \cnode(6,2){8pt}{A}
  \cnode(7.5,2){5pt}{B}
  \cnode(9,2){5pt}{C}
  \cnode(10.5,2){5pt}{D}
  \cnode(12,2){5pt}{E}
  \cnode(13.5,2){5pt}{F}
  \cnode(15,2){5pt}{G}
  \cnode(16.5,2){5pt}{H}
  \cnode(18,2){5pt}{K}  
  \cnode(0,2){5pt}{L}
  \cnode(1.5,2){5pt}{M}
  \cnode(3,2){5pt}{N}
  \cnode(4.5,2){5pt}{O}
  \cnode(-1.5,2){5pt}{P}
  \cnode(19.5,2){5pt}{Q}
  \cnode(7.3,3.1){5pt}{R1}
  \cnode(6.8,3.8){5pt}{R2}
  \cnode(6,4.25){5pt}{R3}  \cnode(6,5.75){5pt}{R32}
  \cnode(5.2,3.8){5pt}{R4}   \cnode(4.6,5.15){5pt}{R42}  \cnode(4,6.5){5pt}{R43}
 \cnode(4.7,3.1){5pt}{R5}
  \cnode(7.3,0.9){5pt}{S1}
  \cnode(6.8,0.2){5pt}{S2}
  \cnode(6,-0.25){5pt}{S3}  \cnode(6,-1.75){5pt}{S32}
  \cnode(5.2,0.2){5pt}{S4}   \cnode(4.6,-1.15){5pt}{S42}  \cnode(4,-2.5){5pt}{S43}
 \cnode(4.7,0.9){5pt}{S5}

   \ncline[nodesep=0pt]{A}{B}\naput[npos=1.1]{$\vphantom{\big(_{p}} \mathrm{St}$}
  \ncline[nodesep=0pt]{B}{C}\naput[npos=1.1]{$\vphantom{\Big(}\phi_{8,91}$}
  \ncline[nodesep=0pt]{C}{D}\naput[npos=1.1]{$\vphantom{\Big(}  \phi_{28,68}$}
  \ncline[nodesep=0pt]{D}{E}\naput[npos=1.1]{$\vphantom{\Big(}\phi_{56,49}$}
  \ncline[nodesep=0pt]{E}{F}\naput[npos=1.1]{$\vphantom{\Big(} \phi_{70,32}$}
  \ncline[nodesep=0pt]{F}{G}\naput[npos=1.1]{$\vphantom{\Big(} \phi_{56,19}$}
  \ncline[nodesep=0pt]{G}{H}\naput[npos=1.1]{$\vphantom{\Big(} \phi_{28,8}$}
  \ncline[nodesep=0pt]{H}{K}\naput[npos=1.1]{$\vphantom{\Big(} \phi_{8,1} $}
  \ncline[nodesep=0pt]{L}{M}\naput[npos=-0.1]{$\vphantom{\Big(} D_{4,\phi_{4,1}} $}  
  \ncline[nodesep=0pt]{M}{N}\naput[npos=-0.1]{$\vphantom{\Big(} D_{4,\phi_{6,6}''} $}
  \ncline[nodesep=0pt]{N}{O}\naput[npos=-0.1]{$\vphantom{\Big(}D_{4,\phi_{4,13}} $}
  \ncline[nodesep=0pt]{O}{A}\naput[npos=-0.1]{$\vphantom{\Big(}D_{4,\varepsilon}  $}
  \ncline[nodesep=0pt]{P}{L}\naput[npos=-0.1]{$\vphantom{\Big(}D_{4,1}  $}
  \ncline[nodesep=0pt]{K}{Q}\naput[npos=1.15]{$\vphantom{\big(_p} 1 $}
  \ncline[nodesep=0pt]{A}{R1}\ncput[npos=1.4]{$\hphantom{aaiaaa}E_8[-\theta^2]$}  
  \ncline[nodesep=0pt]{A}{R2}  \ncput[npos=1.3]{$\hphantom{aaaaaa}E_8[\zeta]$}  
  \ncline[nodesep=0pt]{A}{R3}\nbput[npos=1.3]{$E_7[\mathrm{i}]_{\varepsilon}$}  
  \ncline[nodesep=0pt]{R3}{R32}\nbput[npos=1.5]{$E_7[\mathrm{i}]_{1}$} 
  \ncline[nodesep=0pt]{A}{R4}  \ncput[npos=1.3]{$E_6[\theta]_\varepsilon \hphantom{aaaaaaa}$}
   \ncline[nodesep=0pt]{R4}{R42}  \ncput[npos=1.3]{$E_6[\theta]_{\phi_{2,1}}\hphantom{aaaaaaaa}$}
   \ncline[nodesep=0pt]{R42}{R43}  \ncput[npos=1.3]{$E_6[\theta]_1\hphantom{aaaaaaa}$}
  \ncline[nodesep=0pt]{A}{R5}  \ncput[npos=1.4]{$E_8[\zeta^2]\hphantom{aiaaaa}$}
   \ncline[nodesep=0pt]{A}{S1} \ncput[npos=1.4]{$\hphantom{aaaaaa}E_8[-\theta]$}  
  \ncline[nodesep=0pt]{A}{S2}  \ncput[npos=1.3]{$\hphantom{aaaaaa}E_8[\zeta^4]$}  
  \ncline[nodesep=0pt]{A}{S3} \naput[npos=1.3]{$E_7[-\mathrm{i}]_{\varepsilon}$}  
  \ncline[nodesep=0pt]{S3}{S32}\naput[npos=1.5]{$E_7[-\mathrm{i}]_{1}$} 
  \ncline[nodesep=0pt]{A}{S4}  \ncput[npos=1.3]{$E_6[\theta^2]_\varepsilon \hphantom{aaaaiaaa}$}
   \ncline[nodesep=0pt]{S4}{S42}  \ncput[npos=1.3]{$E_6[\theta^2]_{\phi_{2,1}}\hphantom{aaiaaaaaa}$}
   \ncline[nodesep=0pt]{S42}{S43}  \ncput[npos=1.3]{$E_6[\theta^2]_1\hphantom{aaaaaiaa}$}  \ncline[nodesep=0pt]{A}{S5}   \ncput[npos=1.4]{$E_8[\zeta^3]\hphantom{aaaaaa}$}

\end{pspicture}

\vskip 3.5cm

\caption{Brauer tree of the principal $\Phi_{30}$-block of $E_8$}
\label{E8}

\end{figure}

\end{landscape}

\def\cprime{$'$}

\begin{small}\textsc{Olivier Dudas: Mathematical Institute,
University of Oxford, 24-29 St Giles', Oxford, OX1 3LB, UK.}

\emph{E-mail address}: \texttt{dudas@maths.ox.ac.uk}

\mk

\textsc{Rapha\"el Rouquier: Mathematical Institute,
University of Oxford, 24-29 St Giles', Oxford, OX1 3LB, UK
and Department of Mathematics, UCLA, Box 951555,
Los Angeles, CA 90095-1555, USA.}

\emph{E-mail address}: \texttt{rouquier@maths.ox.ac.uk}

\end{small}


\begin{thebibliography}{10}

\bibitem{BoNee}
A.~B\"okstedt, Marcel~andNeeman.
\newblock Homotopy limits in triangulated categories.
\newblock {\em Compositio Math..}, 86:209--234, 1993.

\bibitem{Bon2}
C.~Bonnaf{\'e}.
\newblock Quasi-isolated elements in reductive groups.
\newblock {\em Comm. Algebra}, 33(7):2315--2337, 2005.

\bibitem{BR1}
C.~Bonnaf{\'e} and R.~Rouquier.
\newblock Cat\'egories d\'eriv\'ees et vari\'et\'es de {D}eligne-{L}usztig.
\newblock {\em Publ. Math. Inst. Hautes \'Etudes Sci.}, (97):1--59, 2003.

\bibitem{Bro1}
M.~Brou{\'e}.
\newblock Isom\'etries de caract\`eres et \'equivalences de {M}orita ou
  d\'eriv\'ees.
\newblock {\em Inst. Hautes \'Etudes Sci. Publ. Math.}, (71):45--63, 1990.

\bibitem{BMa2}
M.~Brou{\'e} and G.~Malle.
\newblock Zyklotomische {H}eckealgebren.
\newblock {\em Ast\'erisque}, (212):\, 119--189, 1993.
\newblock Repr{\'e}sentations unipotentes g{\'e}n{\'e}riques et blocs des
  groupes r{\'e}ductifs finis.

\bibitem{BMR}
M.~Brou\'e, G.~Malle, and R.~Rouquier.
\newblock Complex reflection groups, braid groups, {H}ecke algebras.
\newblock {\em J. Reine Angew. Math.}, 500:127--190, 1998.

\bibitem{BMi2}
M.~Brou{\'e} and J.~Michel.
\newblock Sur certains \'el\'ements r\'eguliers des groupes de {W}eyl et les
  vari\'et\'es de {D}eligne-{L}usztig associ\'ees.
\newblock In {\em Finite reductive groups ({L}uminy, 1994)}, volume 141 of {\em
  Progr. Math.}, pages 73--139.

\bibitem{BMi1}
M.~Brou{\'e} and J.~Michel.
\newblock Blocs et s\'eries de {L}usztig dans un groupe r\'eductif fini.
\newblock {\em J. Reine Angew. Math.}, 395:56--67, 1989.

\bibitem{CaEn2}
M.~Cabanes and M.~Enguehard.
\newblock Local methods for blocks of reductive groups over a finite field.
\newblock In {\em Finite reductive groups ({L}uminy, 1994)}, volume 141 of {\em
  Progr. Math.}, pages 141--163. Birkh\"auser Boston, Boston, MA, 1997.

\bibitem{CaEn}
M.~Cabanes and M.~Enguehard.
\newblock On blocks of finite reductive groups and twisted induction.
\newblock {\em Adv. Math.}, 145(2):189--229, 1999.

\bibitem{Cra1}
D.~Craven.
\newblock On the cohomology of {D}eligne-{L}usztig varieties,
  ar{X}iv:1107.1871v1.
\newblock {\em Preprint}, 2011.

\bibitem{Cra2}
D.~Craven.
\newblock Perverse equivalences and {B}rou\'e's conjecture {I}{I}: the cyclic
  case.
\newblock {\em In preparation}, 2012.

\bibitem{DeLu}
P.~Deligne and G.~Lusztig.
\newblock Representations of reductive groups over finite fields.
\newblock {\em Ann. of Math. (2)}, 103(1):103--161, 1976.

\bibitem{Der}
D.~I. Deriziotis.
\newblock The centralizers of semisimple elements of the {C}hevalley groups
  {$E\sb{7}$} and {$E\sb{8}$}.
\newblock {\em Tokyo J. Math.}, 6(1):191--216, 1983.

\bibitem{DeLi}
D.~I. Deriziotis and M.~W. Liebeck.
\newblock Centralizers of semisimple elements in finite twisted groups of {L}ie
  type.
\newblock {\em J. London Math. Soc. (2)}, 31(1):48--54, 1985.

\bibitem{DM}
F.~Digne and J.~Michel.
\newblock {\em Representations of finite groups of {L}ie type}, volume~21 of
  {\em London Mathematical Society Student Texts}.
\newblock Cambridge University Press, Cambridge, 1991.

\bibitem{DM3}
F.~Digne and J.~Michel.
\newblock {P}arabolic {D}eligne-{L}usztig varieties, arxiv:1110.4863.
\newblock {\em preprint}, 2011.

\bibitem{Du4}
O.~Dudas.
\newblock Coxeter {O}rbits and {B}rauer trees {II}, ar{X}iv:math/1011.5478.
\newblock {\em preprint}, 2010.

\bibitem{Du3}
O.~Dudas.
\newblock Coxeter {O}rbits and {B}rauer trees.
\newblock {\em Adv. Math.}, 229(6):3398--3435, 2012.

\bibitem{Fe}
W.~Feit.
\newblock {\em The representation theory of finite groups}, volume~25 of {\em
  North-Holland Mathematical Library}.
\newblock North-Holland Publishing Co., Amsterdam, 1982.

\bibitem{FlJa}
P.~Fleischmann and I.~Janiszczak.
\newblock The semisimple conjugacy classes of finite groups of {L}ie type
  {$E\sb 6$} and {$E\sb 7$}.
\newblock {\em Comm. Algebra}, 21(1):93--161, 1993.

\bibitem{FoSr}
P.~Fong and B.~Srinivasan.
\newblock Brauer trees in classical groups.
\newblock {\em J. Algebra}, 131(1):179--225, 1990.

\bibitem{Ge3}
M.~Geck.
\newblock Brauer trees of {H}ecke algebras.
\newblock {\em Comm. Algebra}, 20(10):2937--2973, 1992.

\bibitem{GLS}
D.~Gorenstein, R.~Lyons, and R.~Solomon.
\newblock {\em The classification of the finite simple groups. {N}umber 3.
  {P}art {I}. {C}hapter {A}}, volume~40 of {\em Mathematical Surveys and
  Monographs}.
\newblock American Mathematical Society, Providence, RI, 1998.
\newblock Almost simple $K$-groups.

\bibitem{Gr}
J.~A. Green.
\newblock Walking around the {B}rauer {T}ree.
\newblock {\em J. Austral. Math. Soc.}, 17:197--213, 1974.
\newblock Collection of articles dedicated to the memory of Hanna Neumann, VI.

\bibitem{GrJZZa}
F.~Grunewald, A.~Jaikin-Zapirain, and P.~A. Zalesskii.
\newblock Cohomological goodness and the profinite completion of {B}ianchi
  groups.
\newblock {\em Duke Math. J.}, 144(1):53--72, 2008.

\bibitem{Hiss}
G.~Hiss.
\newblock The {B}rauer trees of the {R}ee groups.
\newblock {\em Comm. Algebra}, 19(3), 1991.

\bibitem{HL}
G.~Hiss and F.~L{\"u}beck.
\newblock The {B}rauer trees of the exceptional {C}hevalley groups of types
  {$F\sb 4$} and {$\sp 2\!E\sb 6$}.
\newblock {\em Arch. Math. (Basel)}, 70(1):16--21, 1998.

\bibitem{HLM}
G.~Hiss, F.~L{\"u}beck, and G.~Malle.
\newblock The {B}rauer trees of the exceptional {C}hevalley groups of type
  {$E\sb 6$}.
\newblock {\em Manuscripta Math.}, 87(1):131--144, 1995.

\bibitem{Ke}
B.~Keller.
\newblock On the construction of triangle equivalences.
\newblock In {\em Derived equivalences for group rings}, volume 1685 of {\em
  Lecture Notes in Math.}, pages 155--176. Springer, Berlin, 1998.

\bibitem{Lu}
G.~Lusztig.
\newblock Coxeter orbits and eigenspaces of {F}robenius.
\newblock {\em Invent. Math.}, 38(2):101--159, 1976/77.

\bibitem{Na}
T.~Nakamura.
\newblock A note on the ${K}(\pi ,\,1)$ property of the orbit space of the
  unitary reflection group ${G}(m,\,l,\,n)$.
\newblock {\em Sci. Papers College Arts Sci. Univ. Tokyo}, 33(1):1--6, 1998.

\bibitem{Ri5}
J.~Rickard.
\newblock Finite group actions and \'etale cohomology.
\newblock {\em Inst. Hautes \'Etudes Sci. Publ. Math.}, (80):81--94 (1995),
  1994.

\bibitem{Rou}
R.~Rouquier.
\newblock Complexes de cha\^\i nes \'etales et courbes de {D}eligne-{L}usztig.
\newblock {\em J. Algebra}, 257(2):482--508, 2002.

\bibitem{Ser}
J.-P. Serre.
\newblock {\em Galois cohomology}.
\newblock Springer Monographs in Mathematics. Springer-Verlag, Berlin, english
  edition, 2002.
\newblock Translated from the French by Patrick Ion and revised by the author.

\bibitem{Sha}
J.~Shamash.
\newblock Brauer trees for blocks of cyclic defect in the groups {$G\sb 2(q)$}
  for primes dividing {$q\sp 2\pm q+1$}.
\newblock {\em J. Algebra}, 123(2):378--396, 1989.

\bibitem{Sp}
T.~A. Springer.
\newblock Regular elements of finite reflection groups.
\newblock {\em Invent. Math.}, 25:159--198, 1974.

\end{thebibliography}
\end{document}